\documentclass[12pt]{amsart}

\usepackage{amsmath,amssymb,amsfonts,amsthm,latexsym,graphicx,multirow,hyperref,enumerate}

\newcommand\A{\mathrm{A}}   \newcommand\Aut{\mathrm{Aut}}
 \newcommand\bbF{\mathbb{F}} 
\newcommand\C{\mathrm{C}}  \newcommand\calC{\mathcal{C}} \newcommand\calM{\mathcal{M}} \newcommand\calS{\mathcal{S}} \newcommand\Cay{\mathrm{Cay}} \newcommand\Cen{\mathbf{C}}

   \newcommand\GL{\mathrm{GL}} \newcommand\GO{\mathrm{O}} \newcommand\GU{\mathrm{GU}}

 \newcommand\magma{{\sc Magma} }
\newcommand\N{\mathrm{N}}
\newcommand\Nor{\mathbf{N}}

\newcommand\Pa{\mathrm{P}}   \newcommand\PGL{\mathrm{PGL}} \newcommand\PGO{\mathrm{PGO}} \newcommand\PGU{\mathrm{PGU}} \newcommand\POm{\mathrm{P\Omega}} \newcommand\ppd{\mathrm{ppd}} \newcommand\PSL{\mathrm{PSL}}   \newcommand\PSO{\mathrm{PSO}} \newcommand\PSp{\mathrm{PSp}} \newcommand\PSU{\mathrm{PSU}}

  \newcommand\Soc{\mathrm{Soc}} \newcommand\Sp{\mathrm{Sp}}    \newcommand\Sy{\mathrm{S}}  \newcommand\Sz{\mathrm{Sz}}

\newtheorem{theorem}{Theorem}[section]
\newtheorem{lemma}[theorem]{Lemma}
\newtheorem{proposition}[theorem]{Proposition}
\newtheorem{corollary}[theorem]{Corollary}

\newtheorem{conjecture}[theorem]{Conjecture}
\theoremstyle{definition}

\begin{document}

\title[On cubic graphical regular representations]{On cubic graphical regular representations of finite simple groups}

\author[Xia]{Binzhou Xia}
\address{School of Mathematics and Statistics\\The University of Melbourne\\Parkville, VIC 3010\\Australia}
\email{binzhoux@unimelb.edu.au}


\maketitle

\begin{abstract}
A recent conjecture of the author and Teng Fang states that there are only finitely many finite simple groups with no cubic graphical regular representation. In this paper, we make crucial progress towards this conjecture by giving an affirmative answer for groups of Lie type of large rank.

\textit{Key words:} Cayley graph; cubic graph; graphical regular representation; finite simple groups

\end{abstract}

\section{Introduction}

Given a group $G$ and a subset $S$ of $G$ such that $1\notin S$ and $S^{-1}=S$, where $S^{-1}:=\{g^{-1}\mid g\in S\}$,
the \emph{Cayley graph} $\Cay(G,S)$ of $G$ with \emph{connection set} $S$ is the graph with vertex set $G$ such that two vertices $x$ and $y$ are adjacent
if and only if $yx^{-1}\in S$. It is easy to see that $\Cay(G,S)$ is connected if and only if $S$ generates the group $G$. If one identifies $G$ with its right regular representation, then $G$ is a subgroup of $\Aut(\Cay(G,S))$. We call $\Cay(G,S)$ a \emph{graphical regular representation} (\emph{GRR} for short) of $G$ if $\Aut(\Cay(G,S))=G$.

There was a considerable amount of work on GRRs in the literature, which culminated in~\cite{Godsil1981,Hetzel1976} with the classification of finite groups admitting GRRs. It turns out that most finite groups have GRRs. In particular, every finite nonsolvable group has a GRR. In contrast to the general existence problem of GRRs, not much is known as whether a group has a GRR of prescribed valency. In~\cite{FLWX2002}, Fang, Li, Wang and Xu conjectured that every finite nonabelian simple group has a cubic GRR. Although this conjecture has been shown to be not true, since the authors of~\cite{XF2016} found $\PSL_2(7)$ as a counterexample, it was conjectured that there are only finitely many counterexamples:

\begin{conjecture}\label{conj1}
\emph{(\cite{XF2016})} There are only finitely many finite nonabelian simple groups that have no cubic GRR.
\end{conjecture}

Since cubic graphs have even order, groups of prime order do not have a GRR. Some known results have verified Conjecture~\ref{conj1} for alternating groups~\cite{Godsil1983}, Suzuki groups~\cite{FLWX2002} and projective special linear groups of dimension two~\cite{XF2016}. Very recently, Spiga~\cite{Spiga2018} obtained some sufficient conditions for a cubic Cayley graph of a finite nonabelian simple group to be a GRR. These conditions are quite mild, hence supporting Conjecture~\ref{conj1}. He then made several stronger conjectures on cubic GRRs of nonabelian simple groups, one of which is as follows:

\begin{conjecture}\label{conj2}
\emph{(\cite[Conjecture~1.3]{Spiga2018})} Except for a finite number of cases and for the groups $\PSL_2(q)$, every finite nonabelian simple
group has a cubic GRR whose connection set contains exactly one involution.
\end{conjecture}

In the present paper, we verify Conjectures~\ref{conj1} and~\ref{conj2} in a crucial case by showing that they are true for finite groups of Lie type of rank at least $9$ (Corollary~\ref{cor2}). This in conjunction with previous results reduces the verification of Conjectures~\ref{conj1} and~\ref{conj2} to finite groups of Lie type of rank at most $8$ except $\Sz(q)$ and $\PSL_2(q)$. In fact, our main result (Theorem~\ref{thm1}) is much stronger than the existence of a cubic GRR for the groups described there. It asserts that taken with an element of certain prime order in those groups, a random involution will almost surely give a cubic GRR. To explicitly state the result, we need the notion of primitive prime divisors.

Given positive integers $a$ and $m$, a prime number $r$ is called a \emph{primitive prime divisor} of the pair $(a,m)$ if $r$ divides $a^m-1$ but does not divide $a^i-1$ for any positive integer $i<m$. In other words, a primitive prime divisor of $(a,m)$ is a prime number $r$ such that $a$ has order $m$ in $\mathbb{F}_r^\times$. In particular, if $r$ is a primitive prime divisor of $(a,m)$, then $r>m$. For any positive integers $a$ and $m$, denote the set of primitive prime divisors of $(a,m)$ by $\ppd(a,m)$. By Zsigmondy's theorem (see, for example,~\cite[Theorem~IX.8.3]{Blackburn1982}), $\ppd(a,m)$ is nonempty whenever $a\geqslant2$ and $m\geqslant3$ with $(a,m)\neq(2,6)$.

\begin{theorem}\label{thm1}
Let $G$ be a finite simple classical group and $r\in\ppd(p,ef)$, where $G$ and $e$ are given in Table~$\ref{tab4}$ with $q=p^f$ and $p$ prime. Suppose that $x$ is an element of order $r$ in $G$. Then the probability of a random involution $y$ in $G$ making $\Cay(G,\{x,x^{-1},y\})$ a GRR of $G$ tends to $1$ as $q^n$ tends to infinity.
\end{theorem}

\begin{table}[htbp]
\caption{The pair $(G,e)$ in Theorem~\ref{thm1}}\label{tab4}
\centering
\begin{tabular}{|l|l|l|}
\hline
$G$ & Conditions & $e$\\
\hline
$\PSL_n(q)$ & $n\geqslant9$ & $n$\\
$\PSU_n(q)$ & $n\geqslant5$ odd & $2n$\\
$\PSU_n(q)$ & $n\geqslant6$ even & $2(n-1)$\\
$\PSp_n(q)$ & $n\geqslant10$ even & $n$\\
$\POm_n(q)$ & $n\geqslant9$ odd, $q$ odd & $n-1$\\
$\POm_n^+(q)$ & $n\geqslant14$ even & $n-2$\\
$\POm_n^-(q)$ & $n\geqslant14$ even & $n$\\
\hline
\end{tabular}
\end{table}

Here are some corollaries of Theorem~\ref{thm1}.

\begin{corollary}\label{cor1}
With at most finitely many exceptions, every group in Table~$\ref{tab4}$ has a cubic GRR whose connection set contains exactly one involution.
\end{corollary}

\begin{corollary}\label{cor2}
With at most finitely many exceptions, every finite simple group of Lie type of rank at least $9$ has a cubic GRR whose connection set contains exactly one involution.
\end{corollary}

\begin{corollary}\label{cor3}
There exists a constant $N$ such that every finite simple group of Lie type of rank at least $N$ has a cubic GRR whose connection set contains exactly one involution.
\end{corollary}

Note that every GRR of a finite simple group is connected, for otherwise its full automorphism group will be a wreath product. This means that if $\Cay(G,S)$ is a GRR of a finite simple group $G$, then $S$ is necessarily a generating set of $G$. In Theorem~\ref{thm1}, the connection set consists of an involution and element of prime order together with its inverse. Then since the connection set must be a generating set to form a GRR of $G$, Theorem~\ref{thm1} also shows that the probability of a random involution $y$ generating $G$ with $x$ tends to $1$ as $q^n$ tends to infinity. This byproduct, as stated in Proposition~\ref{prop1}, is actually the asymptotic probabilistic version of a recent result of King~\cite{King2017}. For the probabilistic generation problem of finite simple groups, the reader is referred to the work of Liebeck and Shalev~\cite{LS1996}.

Finally we remark that Theorem~\ref{thm1} may fail for finite groups of Lie type of small rank. For example, \cite[Proposition~1.5]{XF2016} shows that the connection set for a cubic GRR of $\PSL_2(q)$ must consist of three involutions.


\section{Preliminaries}

We first prove some elementary results that are needed in the sequel.

\begin{lemma}\label{lem8}
For each integer $n\geqslant7$ we have $(n+2)!<\min\{2^{\frac{n^2}{4}-\frac{n}{8}+8},3^{\frac{n^2}{4}-\frac{5n}{8}+4}\}$.
\end{lemma}

\begin{proof}
It is direct to verify that
\begin{equation}\label{eq3}
9!<2^{\frac{155}{8}}\quad\text{and}\quad9!<3^{\frac{95}{8}}.
\end{equation}
Note that
\[
k+2<2^\frac{4k-3}{8}\quad\text{and}\quad k+2<3^\frac{4k-7}{8}
\]
for every $k\geqslant8$. Then for every integer $n\geqslant8$ we have
\[
\prod_{k=8}^n(k+2)<\prod_{k=8}^n2^\frac{4k-3}{8}\quad\text{and}\quad\prod_{k=8}^n(k+2)<\prod_{k=8}^n3^\frac{4k-7}{8}.
\]
This together with~\eqref{eq3} implies that
\[
(n+2)!=9!\prod_{k=8}^n(k+2)<2^{\frac{155}{8}}\prod_{k=8}^n2^\frac{4k-3}{8}=2^{\frac{n^2}{4}-\frac{n}{8}+8}
\]
and
\[
(n+2)!=9!\prod_{k=8}^n(k+2)<3^{\frac{95}{8}}\prod_{k=8}^n3^\frac{4k-7}{8}=3^{\frac{n^2}{4}-\frac{5n}{8}+4}
\]
for all integers $n\geqslant7$, completing the proof.
\end{proof}

\begin{lemma}\label{lem9}
Suppose that $s>0$ is a constant. Then there exists a constant $M$ such that $n^{\log_2(n)}<q^{sn}$ for all prime powers $q$ and positive integers $n$ with $q^n>M$.
\end{lemma}

\begin{proof}
Since
\[
\lim_{n\rightarrow\infty}\frac{(\log_2(n))^2}{sn}=0,
\]
there exists a positive integer $N$ such that $(\log_2(n))^2<sn$ for all $n>N$. Take
\[
\mu=\max\{n^\frac{\log_2(n)}{sn}\mid1\leqslant n\leqslant N\}
\]
and $M=\mu^N$. Let $q$ be a prime power and $n$ be a positive integer such that $q^n>M$. If $n>N$, then we deduce from $(\log_2(n))^2<sn$ that
\[
n^{\log_2(n)}=2^{(\log_2(n))^2}<2^{sn}\leqslant q^{sn}.
\]
If $n\leqslant N$, then $q^N\geqslant q^n>M=\mu^N$ and so $q>\mu\geqslant n^\frac{\log_2(n)}{sn}$, which again leads to $n^{\log_2(n)}<q^{sn}$. This completes the proof.
\end{proof}

For any integer $n$ denote the set of prime divisors of $n$ by $\pi(n)$.

\begin{lemma}\label{lem10}
If $n\geqslant2$ then $|\pi(n)|\leqslant\log_2(n)$.
\end{lemma}

\begin{proof}
Let $p_1<p_2<\dots<p_k$ be the prime divisors of $n$, where $k=|\pi(n)|$. Then
\[
n\geqslant\prod_{i=1}^kp_i\geqslant\prod_{i=1}^k2=2^k.
\]
This shows that $k\leqslant\log_2(n)$, as desired.
\end{proof}
The undefined group theoretic notation in this paper is either standard or follows~\cite{KL1990}. In particular, we use the notation $\calC_i$ with $1\leqslant i\leqslant8$ and $\calS$ for the classes of maximal subgroups of finite simple classical groups described in~\cite{KL1990}. The following lemma is from~\cite[Proposition~6.4]{King2017}. (There are errors in the expressions of $|\Nor_G(\langle x\rangle)|$ for $G=\POm_n^+(q)$ and $G=\POm_n^-(q)$ in~\cite[Table~9]{King2017}, which have been corrected in Table~\ref{tab8} below.)

\begin{lemma}\label{lem7}
Let $G$ be a finite simple classical group and $r\in\ppd(p,ef)$, where $G$ and $e$ are given in Table~$\ref{tab8}$ with $q=p^f$, $p$ prime and $\ppd(p,ef)\neq\emptyset$. Then for any $x\in G$ of order $r$, $|\Nor_G(\langle x\rangle)|$ lies in Table~$\ref{tab8}$.
\end{lemma}

\begin{table}[htbp]
\caption{The triple $(G,e,|\Nor_G(\langle x\rangle)|)$ in Lemma~\ref{lem7}}\label{tab8}
\centering
\begin{tabular}{|l|l|l|l|}
\hline
$G$ & Conditions & $e$ & $|\Nor_G(\langle x\rangle)|$\\
\hline
\multirow{2}*{$\PSL_n(q)$} & \multirow{2}*{$n\geqslant2$, $(n,q)\neq(2,2)$ or $(2,3)$} & \multirow{2}*{$n$} & \multirow{2}*{$\frac{n(q^n-1)}{(q-1)\gcd(n,q-1)}$}\\
 & & & \\
\multirow{2}*{$\PSU_n(q)$} & \multirow{2}*{$n\geqslant3$ odd, $(n,q)\neq(3,2)$} & \multirow{2}*{$2n$} & \multirow{2}*{$\frac{n(q^n+1)}{(q+1)\gcd(n,q+1)}$}\\
 & & & \\
\multirow{2}*{$\PSU_n(q)$} & \multirow{2}*{$n\geqslant4$ even} & \multirow{2}*{$2(n-1)$} & \multirow{2}*{$\frac{(n-1)(q^{n-1}+1)}{\gcd(n,q+1)}$}\\
 & & & \\
\multirow{2}*{$\PSp_n(q)$} & \multirow{2}*{$n\geqslant4$ even, $(n,q)\neq(4,2)$} & \multirow{2}*{$n$} & \multirow{2}*{$\frac{n(q^\frac{n}{2}+1)}{\gcd(2,q-1)}$}\\
 & & & \\
\multirow{2}*{$\POm_n(q)$} & \multirow{2}*{$n\geqslant7$ odd, $q$ odd} & \multirow{2}*{$n-1$} & \multirow{2}*{$\frac{(n-1)(q^\frac{n-1}{2}+1)}{2}$}\\
 & & & \\
\multirow{2}*{$\POm_n^+(q)$} & \multirow{2}*{$n\geqslant8$ even} & \multirow{2}*{$n-2$} &\multirow{2}*{$\frac{2(n-2)(q^{\frac{n}{2}-1}+1)(q+1)}{\gcd(2,q-1)^2|\PSO_n^+(q)/G|}$}\\
 & & & \\
\multirow{2}*{$\POm_n^-(q)$} & \multirow{2}*{$n\geqslant8$ even} & \multirow{2}*{$n$} &\multirow{2}*{$\frac{n(q^\frac{n}{2}+1)}{\gcd(2,q-1)^2|\PSO_n^-(q)/G|}$}\\
 & & & \\
\hline
\end{tabular}
\end{table}

For a finite group $G$ denote the number of involutions in $G$ by $i_2(G)$.

\begin{lemma}\label{lem4}
\emph{(\cite[Proposition~3.1]{King2017})} Let $G$ be a finite simple classical group. Then $i_2(G)\geqslant i(G)$ with $i(G)$ given in Table~$\ref{tab2}$.
\end{lemma}

\begin{table}[htbp]
\caption{The pair $(G,i(G))$ in Lemma~\ref{lem4}}\label{tab2}
\centering
\begin{tabular}{|l|l|l|}
\hline
$G$ & Conditions & $i(G)$\\
\hline
$\PSL_n(q)$ & $n\geqslant2$, $(n,q)\neq(2,2)$ or $(2,3)$ & $\frac{1}{8}q^{\lfloor\frac{n^2}{2}\rfloor}$\\
$\PSU_n(q)$ & $n\geqslant3$, $(n,q)\neq(3,2)$ & $\frac{1}{8}q^{\lfloor\frac{n^2}{2}\rfloor}$\\
$\PSp_n(q)$ & $n\geqslant4$ even, $(n,q)\neq(4,2)$ & $\frac{1}{2}q^{\frac{n^2}{4}+\frac{n}{2}}$\\
$\POm_n(q)$ & $n\geqslant7$ odd, $q$ odd & $\frac{1}{2}q^{\frac{n^2-1}{4}}$\\
$\POm_n^+(q)$ & $n\geqslant8$ even & $\frac{1}{8}q^{\frac{n^2}{4}-1}$\\
$\POm_n^-(q)$ & $n\geqslant8$ even & $\frac{1}{8}q^{\frac{n^2}{4}-1}$\\
\hline
\end{tabular}
\end{table}

\begin{lemma}\label{lem2}
Let $G$ be a finite simple classical group. Then $i_2(\Aut(G))<j(G)$ with $j(G)$ given in Table~$\ref{tab5}$.
\end{lemma}

\begin{table}[htbp]
\caption{The pair $(G,j(G))$ in Lemma~\ref{lem2}}\label{tab5}
\centering
\begin{tabular}{|l|l|l|l|l|}
\hline
$G$ & Conditions & $\dim(Y)$ & $N$ & $j(G)$\\
\hline
$\PSL_n(q)$ & $n\geqslant2$, $(n,q)\neq(2,2)$ or $(2,3)$ & $n^2-1$ & $n(n-1)/2$ & $3q^{\frac{n^2}{2}+\frac{n}{2}-1}$\\
$\PSU_n(q)$ & $n\geqslant3$, $(n,q)\neq(3,2)$ & $n^2-1$ & $n(n-1)/2$ & $3q^{\frac{n^2}{2}+\frac{n}{2}-1}$\\
$\PSp_n(q)$ & $n\geqslant4$ even, $(n,q)\neq(4,2)$ & $n(n+1)/2$ & $n^2/4$ & $3q^{\frac{n^2}{4}+\frac{n}{2}}$\\
$\POm_n(q)$ & $n\geqslant7$ odd, $q$ odd & $n(n-1)/2$ & $(n-1)^2/4$ & $3q^{\frac{n^2-1}{4}}$\\
$\POm_n^+(q)$ & $n\geqslant8$ even & $n(n-1)/2$ & $n(n-2)/4$ & $3q^{\frac{n^2}{4}}$\\
$\POm_n^-(q)$ & $n\geqslant8$ even & $n(n-1)/2$ & $n(n-2)/4$ & $3q^{\frac{n^2}{4}}$\\
\hline
\end{tabular}
\end{table}

\begin{proof}
Let $Y$ and $N$ be as in~\cite[Proposition~1.3]{LLS2002}. Then $\dim(Y)$ and $N$ lie in Table~\ref{tab5}, from which we see that $j(G)=3q^{\dim(Y)-N}$. By~\cite[Proposition~1.3]{LLS2002},
\begin{align*}
i_2(\Aut(G))&<2(q^{\dim(Y)-N}+q^{\dim(Y)-N-1})\\
&\leqslant2\left(q^{\dim(Y)-N}+\frac{q^{\dim(Y)-N}}{2}\right)=3q^{\dim(Y)-N}.
\end{align*}
Thus we have $i_2(\Aut(G))<j(G)$, as the lemma asserts.
\end{proof}

The next lemma is from~\cite[Theorem~2.2]{LS1996}.

\begin{lemma}\label{lem6}
There is an absolute constant $c$ such that, if $G$ is a finite simple classical group and $M$ is any maximal subgroup of $G$, then
\[
\frac{i_2(M)}{i_2(G)}<c|G{:}M|^{-\frac{2}{5}}.
\]
\end{lemma}

Recall that the $\calS$-subgroups of a finite simple classical group are all almost simple.

\begin{lemma}\label{lem1}
Let $G$ and $e$ be as in Table~$\ref{tab11}$ and $r\in\ppd(p,ef)$, where $q=p^f$ with $p$ prime. Then the following hold:
\begin{enumerate}[{\rm(a)}]
\item The number of conjugacy classes of $\calS$-subgroups $M$ of $G$ with order divisible by $r$ and socle an alternating group is at most $a(G)$, where $a(G)$ lies in Table~$\ref{tab11}$; all such $\calS$-subgroups $M$ satisfy $i_2(M)<(n+2)!$.
\item The number of conjugacy classes of $\calS$-subgroups $M$ of $G$ with order divisible by $r$ and socle not an alternating group is at most $b(G)$, where $b(G)$ lies in Table~$\ref{tab11}$; all such $\calS$-subgroups $M$ satisfy $i_2(M)<q^{2n+4}$.
\end{enumerate}
\end{lemma}

{\small
\begin{table}[htbp]
\caption{The tuple $(G,e,a(G),b(G))$ in Lemma~\ref{lem1}}\label{tab11}
\centering
\begin{tabular}{|l|l|l|l|l|}
\hline
$G$ & Conditions & $e$ & $a(G)$ & $b(G)$\\
\hline
$\PSL_n(q)$ & $n\geqslant7$ & $n$ & $0$ & $\left(n^2+\frac{21}{4}n-1\right)\gcd(n,q-1)$\\
$\PSU_n(q)$ & $n\geqslant7$ odd & $2n$ & $0$ & $3\gcd(n,q+1)$\\
$\PSU_n(q)$ & $n\geqslant8$ even & $2(n-1)$ & $0$ & $3\gcd(n,q+1)$\\
$\PSp_n(q)$ & $n\geqslant8$ even & $n$ & $\gcd(2,q-1)$ & $\left(n^2+\frac{21}{4}n-1\right)\gcd(2,q-1)$\\
$\POm_n(q)$ & $n\geqslant9$ odd, $q$ odd & $n-1$ & $2$ & $2n^2+12n+8$\\
$\POm_n^+(q)$ & $n\geqslant10$ even & $n-2$ & $|\GO_n^+(q)|/|G|$ & $\left(\frac{1}{4}n+9\right)|\GO_n^+(q)|/|G|$\\
$\POm_n^-(q)$ & $n\geqslant8$ even & $n$ & $|\GO_n^-(q)|/|G|$ & $\left(n^2+\frac{21}{4}n-1\right)|\GO_n^-(q)|/|G|$\\
\hline
\end{tabular}
\end{table}
}

\begin{proof}
Let $M$ be an $\calS$-subgroup of $G$ with order divisible by $r$ such that $\Soc(M)=\A_m$ for some $m\geqslant5$. Then since $r$ divides $|M|$, we have $r\leqslant m$. Moreover, $r\in\ppd(p,ef)$ implies that $r>ef\geqslant e$. If $m\leqslant10$, then $r\leqslant7$ and so $e<7$, contrary to the conditions in Table~\ref{tab11}. Hence $m\geqslant11$. Then part~(a) of the lemma follows from~\cite[Proposition~5.1]{King2017} and the proof of~\cite[Proposition~6.5]{King2017}. Moreover, part(b) of the lemma follows from~\cite[Corollary~4.3~and~Proposition~5.1]{King2017}.
\end{proof}

\section{Probability for $(2,r)$-generation}

The main result of this section is the following proposition, which will be proved through Lemmas~\ref{lem11}--\ref{lem12}.

\begin{proposition}\label{prop1}
Let $G$ be a finite simple classical group and $r\in\ppd(p,ef)$, where $G$ and $e$ are given in Table~$\ref{tab14}$ with $q=p^f$ and $p$ prime. Suppose that $x$ is an element of order $r$ in $G$. Then the probability of a random involution $y$ in $G$ satisfying $G=\langle x,y\rangle$ tends to $1$ as $q^n$ tends to infinity.
\end{proposition}

\begin{table}[htbp]
\caption{The pair $(G,e)$ in Proposition~\ref{prop1}}\label{tab14}
\centering
\begin{tabular}{|l|l|l|}
\hline
$G$ & Conditions & $e$\\
\hline
$\PSL_n(q)$ & $n\geqslant4$, $(n,q)\neq(6,2)$ & $n$\\
$\PSU_n(q)$ & $n\geqslant3$ odd, $(n,q)\neq(3,2)$ & $2n$\\
$\PSU_n(q)$ & $n\geqslant4$ even, $(n,q)\neq(4,2)$ & $2(n-1)$\\
$\PSp_n(q)$ & $n\geqslant8$ even & $n$\\
$\POm_n(q)$ & $n\geqslant9$ odd, $q$ odd & $n-1$\\
$\POm_n^+(q)$ & $n\geqslant10$ even & $n-2$\\
$\POm_n^-(q)$ & $n\geqslant8$ even & $n$\\
\hline
\end{tabular}
\end{table}

Note that for any group $G$ and $x\in G$, if $\calM$ is the set of maximal subgroups of $G$ containing $x$, then the probability of a random involution $y$ in $G$ satisfying $G=\langle x,y\rangle$ is bounded below by
\begin{equation}\label{eq8}
1-\sum_{M\in\calM}\frac{i_2(M)}{i_2(G)}.
\end{equation}

\begin{lemma}\label{lem11}
Let $n\geqslant4$ and $q=p^f$ with prime $p$ such that $(n,q)\neq(6,2)$. Let $G=\PSL_n(q)$ and $r\in\ppd(p,nf)$. Suppose that $x$ is an element of order $r$ in $G$. Then the probability of a random involution $y$ in $G$ satisfying $G=\langle x,y\rangle$ tends to $1$ as $q^n$ tends to infinity.
\end{lemma}

\begin{proof}
Let $\calM$ be the set of maximal subgroups of $G$ containing $x$ and let $\mu$ be a set of $G$-conjugacy class representatives for $\calM$. Then by~\cite[Corollary~6.2]{King2017},
\begin{equation}\label{eq17}
\sum_{M\in\calM}\frac{i_2(M)}{i_2(G)}=\sum_{M\in\mu}\frac{|\Nor_G(\langle x\rangle)|i_2(M)}{|\Nor_M(\langle x\rangle)|i_2(G)}.
\end{equation}
Note from Lemmas~\ref{lem7} and~\ref{lem4} that
\[
|\Nor_G(\langle x\rangle)|=\frac{n(q^n-1)}{(q-1)\gcd(n,q-1)}\leqslant\frac{n(q^n-1)}{q-1}<2nq^{n-1}
\]
and
\[
i_2(G)\geqslant\frac{q^{\lfloor\frac{n^2}{2}\rfloor}}{8}\geqslant\frac{q^{\frac{n^2}{2}-\frac{1}{2}}}{8}.
\]

By~\cite[Proposition~4.1]{King2017}, every $M\in\mu\setminus\calS$ lies in Table~\ref{tab1}, where $c(M)$ is an upper bound on the number of $G$-conjugacy classes of each type, $k(M)$ is an upper bound on $i_2(M)$, $N(M)$ is a lower bound on $|\Nor_M(\langle x\rangle)|$ and $d(M)$ is an upper bound on $c(M)k(M)/N(M)$.

\begin{table}[htbp]
\caption{The tuple $(M,c(M),k(M),N(M),d(M))$ in the proof of Lemma~\ref{lem11}}\label{tab1}
\centering
\begin{tabular}{|l|l|l|l|l|l|l|}
\hline
Class & Type of $M$ & Conditions & $c(M)$ & $k(M)$ & $N(M)$ & $d(M)$\\
\hline
$\calC_3$ & $\GL_k(q^t).\C_t$ & $n=kt$, & $1$ & $4q^{\frac{n^2}{2t}+\frac{n}{t}-1}$ & $|\Nor_G(\langle x\rangle)|$ & $\frac{4q^{\frac{n^2}{4}+\frac{n}{2}-1}}{|\Nor_G(\langle x\rangle)|}$\\
 &  & $t$ prime &  &  &  &  \\
$\calC_6$ & $\C_2^{2k}.\Sp_{2k}(2)$ & $n=2^k$, & $n$ & $n^{2\log_2(n)+3}$ & $n+1$ & $n^{2\log_2(n)+3}$\\
 &  & $q$ odd, &  &  &  &  \\
 &  & $r=n+1$ &  &  &  &  \\
$\calC_8$ & $\PSp_n(q)$ & $n$ even & $n/2$ & $3q^{\frac{n^2}{4}+\frac{n}{2}}$ & $nq^\frac{n}{2}/2$ & $3q^\frac{n^2}{4}$\\
$\calC_8$ & $\PSO_n^-(q)$ & $n$ even, & $n/2$ & $3q^\frac{n^2}{4}$ & $nq^\frac{n}{2}/4$ & $q^\frac{n^2}{4}$\\
 &  & $q$ odd &  &  &  &  \\
$\calC_8$ & $\PSU_n(q^\frac{1}{2})$ & $n$ odd, & $q^\frac{1}{2}$ & $3q^{\frac{n^2}{4}+\frac{n}{4}-\frac{1}{2}}$ & $q^{\frac{n}{2}-\frac{1}{2}}/2$ & $3q^\frac{n^2}{4}$\\
 &  & $q$ square &  &  &  &  \\
\hline
\end{tabular}
\end{table}

It follows from Table~\ref{tab1} and Lemma~\ref{lem10} that
\[
\sum_{M\in\mu\cap\calC_3}\frac{|\Nor_G(\langle x\rangle)|i_2(M)}{|\Nor_M(\langle x\rangle)|i_2(G)}
\leqslant\sum_{t\in\pi(n)}\frac{4q^{\frac{n^2}{4}+\frac{n}{2}-1}}{i_2(G)}
\leqslant\frac{4\log_2(n)q^{\frac{n^2}{4}+\frac{n}{2}-1}}{i_2(G)}<\frac{32\log_2(n)}{q^\frac{n}{8}}.
\]
Moreover, we deduce from Table~\ref{tab1} that
\[
\sum_{M\in\mu\cap\calC_6}\frac{|\Nor_G(\langle x\rangle)|i_2(M)}{|\Nor_M(\langle x\rangle)|i_2(G)}
\leqslant\frac{n^{2\log_2(n)+3}|\Nor_G(\langle x\rangle)|}{i_2(G)}
<\frac{2n^{2\log_2(n)+4}q^{n-1}}{i_2(G)}<\frac{16n^{2\log_2(n)+4}}{q^n},
\]
\begin{align*}
\sum_{M\in\mu\cap\calC_8}\frac{|\Nor_G(\langle x\rangle)|i_2(M)}{|\Nor_M(\langle x\rangle)|i_2(G)}
&\leqslant\frac{3q^{\frac{n^2}{4}}|\Nor_G(\langle x\rangle)|}{i_2(G)}
+\frac{q^{\frac{n^2}{4}}|\Nor_G(\langle x\rangle)|}{i_2(G)}
+\frac{3q^{\frac{n^2}{4}}|\Nor_G(\langle x\rangle)|}{i_2(G)}\\
&=\frac{7q^{\frac{n^2}{4}}|\Nor_G(\langle x\rangle)|}{i_2(G)}
<\frac{14nq^{\frac{n^2}{4}+n-1}}{i_2(G)}
\leqslant\frac{112n}{q^\frac{n}{8}},
\end{align*}
and hence
\begin{align}\label{eq18}
\sum_{M\in\mu\setminus\calS}\frac{|\Nor_G(\langle x\rangle)|i_2(M)}{|\Nor_M(\langle x\rangle)|i_2(G)}
&<\frac{32\log_2(n)}{q^\frac{n}{8}}
+\frac{16n^{2\log_2(n)+4}}{q^n}
+\frac{112n}{q^\frac{n}{8}}\\\nonumber
&<\frac{16n}{q^\frac{n}{8}}+\frac{16n^{4\log_2(n)}}{q^n}+\frac{112n}{q^\frac{n}{8}}
=\frac{16n^{4\log_2(n)}}{q^n}+\frac{128n}{q^\frac{n}{8}}.
\end{align}

According to~\cite[Tables~8.9,~8.19,~8.25~and~8.36]{BHR2013}, for $n\leqslant7$ there exist absolute constants $C_1$ and $C_2$ such that the number of conjugacy classes of $\calS$-subgroups $M$ of $G$ with order divisible by $r$ is at most $C_1$ and all such $\calS$-subgroups $M$ satisfy $i_2(M)<C_2$. This implies that for $n\leqslant7$,
\begin{align*}
\sum_{M\in\mu\cap\calS}\frac{|\Nor_G(\langle x\rangle)|i_2(M)}{|\Nor_M(\langle x\rangle)|i_2(G)}
&<\sum_{M\in\mu\cap\calS}\frac{|\Nor_G(\langle x\rangle)|i_2(M)}{i_2(G)}\\
&<\frac{C_1C_2|\Nor_G(\langle x\rangle)|}{i_2(G)}
<\frac{2C_1C_2nq^{n-1}}{i_2(G)}
<\frac{16C_1C_2n}{q^\frac{n}{2}}
<\frac{C_1C_2n^4}{q^\frac{n}{2}}.
\end{align*}
For $n\geqslant8$, we derive from Lemma~\ref{lem1} that
\[
\sum_{M\in\mu\cap\calS}i_2(M)<\left(n^2+\frac{21}{4}n-1\right)\gcd(n,q-1)q^{2n+4}<3n^3q^{2n+4},
\]
whence
\[
\sum_{M\in\mu\cap\calS}\frac{|\Nor_G(\langle x\rangle)|i_2(M)}{|\Nor_M(\langle x\rangle)|i_2(G)}
<\sum_{M\in\mu\cap\calS}\frac{2nq^{n-1}i_2(M)}{i_2(G)}\\
<\frac{6n^4q^{3n+3}}{i_2(G)}<\frac{48n^4}{q^\frac{n}{2}}.
\]
Therefore, we have for all $n\geqslant4$ that
\begin{equation}\label{eq4}
\sum_{M\in\mu\cap\calS}\frac{|\Nor_G(\langle x\rangle)|i_2(M)}{|\Nor_M(\langle x\rangle)|i_2(G)}<\frac{C_0n^4}{q^\frac{n}{2}},
\end{equation}
where $C_0=\max\{C_1C_2,48\}$.

Combining~\eqref{eq18} and~\eqref{eq4} we conclude that
\[
\sum_{M\in\mu}\frac{|\Nor_G(\langle x\rangle)|i_2(M)}{|\Nor_M(\langle x\rangle)|i_2(G)}
<\frac{16n^{4\log_2(n)}}{q^n}+\frac{128n}{q^\frac{n}{8}}+\frac{C_0n^4}{q^\frac{n}{2}}.
\]
By Lemma~\ref{lem9}, for sufficiently large $q^n$ we have $n^{4\log_2(n)}<q^\frac{n}{2}$ and so
\[
\sum_{M\in\mu}\frac{|\Nor_G(\langle x\rangle)|i_2(M)}{|\Nor_M(\langle x\rangle)|i_2(G)}
<\frac{16}{q^\frac{n}{2}}+\frac{128n}{q^\frac{n}{8}}+\frac{C_0n^4}{q^\frac{n}{2}}
\leqslant\frac{16}{(q^n)^{\frac{1}{2}}}+\frac{128\log_2(q^n)}{(q^n)^{\frac{1}{8}}}+\frac{C_0(\log_2(q^n))^4}{(q^n)^{\frac{1}{2}}}.
\]
This in conjunction with~\eqref{eq17} shows that~\eqref{eq8} tends to $1$ as $q^n$ tends to infinity. Then the conclusion of the lemma follows readily.
\end{proof}

\begin{lemma}
Let $n\geqslant3$ be odd and $q=p^f$ with prime $p$ such that $(n,q)\neq(3,2)$. Let $G=\PSU_n(q)$ and $r\in\ppd(p,2nf)$. Suppose that $x$ is an element of order $r$ in $G$. Then the probability of a random involution $y$ in $G$ satisfying $G=\langle x,y\rangle$ tends to $1$ as $q^n$ tends to infinity.
\end{lemma}

\begin{proof}
Let $\calM$ be the set of maximal subgroups of $G$ containing $x$ and let $\mu$ be a set of $G$-conjugacy class representatives for $\calM$. Then by~\cite[Corollary~6.2]{King2017},
\begin{equation}\label{eq1}
\sum_{M\in\calM}\frac{i_2(M)}{i_2(G)}=\sum_{M\in\mu}\frac{|\Nor_G(\langle x\rangle)|i_2(M)}{|\Nor_M(\langle x\rangle)|i_2(G)}.
\end{equation}
Note from Lemmas~\ref{lem7} and~\ref{lem4} that
\[
|\Nor_G(\langle x\rangle)|=\frac{n(q^n+1)}{(q+1)\gcd(n,q+1)}\leqslant\frac{n(q^n+1)}{q+1}<nq^{n-1}
\]
and
\[
i_2(G)\geqslant\frac{q^{\lfloor\frac{n^2}{2}\rfloor}}{8}=\frac{q^{\frac{n^2}{2}-\frac{1}{2}}}{8}.
\]

By~\cite[Proposition~4.1]{King2017}, the groups $M\in\mu\setminus\calS$ have type $\GU_k(q^t).\C_t$, where $n=kt$ and $t\geqslant3$ is prime. Moreover, for each odd prime divisor $t$ of $n$ there is exactly one $G$-conjugacy class of groups $M$ of type $\GU_k(q^t).\C_t$, and they satisfy
\[
i_2(M)<2q^{\frac{n^2}{2t}+\frac{n}{2}-1}\quad\text{and}\quad|\Nor_M(\langle x\rangle)|=|\Nor_G(\langle x\rangle)|.
\]
It then follows from Lemma~\ref{lem10} that
\begin{equation}\label{eq2}
\sum_{M\in\mu\setminus\calS}\frac{|\Nor_G(\langle x\rangle)|i_2(M)}{|\Nor_M(\langle x\rangle)|i_2(G)}
<\sum_{t\in\pi(n)}\frac{2q^{\frac{n^2}{2t}+\frac{n}{2}-1}}{i_2(G)}
\leqslant\frac{2\log_2(n)q^{\frac{n^2}{6}+\frac{n}{2}-1}}{i_2(G)}
\leqslant\frac{16\log_2(n)}{q^\frac{2n}{3}}.
\end{equation}

For $n\leqslant7$, we see from~\cite[Tables~8.6,~8.21~and~8.57]{BHR2013} that there exist absolute constants $C_1$ and $C_2$ such that the number of conjugacy classes of $\calS$-subgroups $M$ of $G$ with order divisible by $r$ is at most $C_1$ and all such $\calS$-subgroups $M$ satisfy $i_2(M)<C_2$. This implies that for $n\leqslant7$,
\begin{align*}
\sum_{M\in\mu\cap\calS}\frac{|\Nor_G(\langle x\rangle)|i_2(M)}{|\Nor_M(\langle x\rangle)|i_2(G)}
&<\sum_{M\in\mu\cap\calS}\frac{|\Nor_G(\langle x\rangle)|i_2(M)}{i_2(G)}\\
&<\frac{C_1C_2|\Nor_G(\langle x\rangle)|}{i_2(G)}
<\frac{C_1C_2nq^{n-1}}{i_2(G)}
<\frac{8C_1C_2n}{q^\frac{2n}{3}}
<\frac{3C_1C_2n^2}{q^\frac{2n}{3}}.
\end{align*}
For $n\geqslant9$, we derive from Lemma~\ref{lem1} that
\[
\sum_{M\in\mu\cap\calS}i_2(M)<3\gcd(n,q+1)q^{2n+4}\leqslant3nq^{2n+4}
\]
and hence
\[
\sum_{M\in\mu\cap\calS}\frac{|\Nor_G(\langle x\rangle)|i_2(M)}{|\Nor_M(\langle x\rangle)|i_2(G)}
<\sum_{M\in\mu\cap\calS}\frac{nq^{n-1}i_2(M)}{i_2(G)}<\frac{3n^2q^{3n+3}}{i_2(G)}<\frac{2n^2}{q^\frac{2n}{3}}.
\]
Therefore, we have for all odd $n\geqslant3$ that
\begin{equation}\label{eq6}
\sum_{M\in\mu\cap\calS}\frac{|\Nor_G(\langle x\rangle)|i_2(M)}{|\Nor_M(\langle x\rangle)|i_2(G)}<\frac{C_0n^2}{q^\frac{2n}{3}},
\end{equation}
where $C_0=\max\{3C_1C_2,2\}$.

Combining~\eqref{eq2} and~\eqref{eq6} we conclude that
\[
\sum_{M\in\mu}\frac{|\Nor_G(\langle x\rangle)|i_2(M)}{|\Nor_M(\langle x\rangle)|i_2(G)}
<\frac{16\log_2(n)}{q^\frac{2n}{3}}+\frac{C_0n^2}{q^\frac{2n}{3}}
<\frac{4n^2}{q^\frac{2n}{3}}+\frac{C_0n^2}{q^\frac{2n}{3}}
\leqslant\frac{(4+C_0)(\log_2(q^n))^2}{(q^n)^{\frac{2}{3}}}.
\]
This together with~\eqref{eq1} shows that~\eqref{eq8} tends to $1$ as $q^n$ tends to infinity. Then the conclusion of the lemma follows readily.
\end{proof}

\begin{lemma}
Let $n\geqslant4$ be even and $q=p^f$ with prime $p$ such that $(n,q)\neq(4,2)$. Let $G=\PSU_n(q)$ and $r\in\ppd(p,2(n-1)f)$. Suppose that $x$ is an element of order $r$ in $G$. Then the probability of a random involution $y$ in $G$ satisfying $G=\langle x,y\rangle$ tends to $1$ as $q^n$ tends to infinity.
\end{lemma}

\begin{proof}
Let $\calM$ be the set of maximal subgroups of $G$ containing $x$ and let $\mu$ be a set of $G$-conjugacy class representatives for $\calM$. Then by~\cite[Corollary~6.2]{King2017},
\begin{equation}\label{eq7}
\sum_{M\in\calM}\frac{i_2(M)}{i_2(G)}=\sum_{M\in\mu}\frac{|\Nor_G(\langle x\rangle)|i_2(M)}{|\Nor_M(\langle x\rangle)|i_2(G)}.
\end{equation}
Note from Lemmas~\ref{lem7} and~\ref{lem4} that
\[
|\Nor_G(\langle x\rangle)|=\frac{(n-1)(q^{n-1}+1)}{\gcd(n,q+1)}\leqslant(n-1)(q^{n-1}+1)<nq^{n-1}
\]
and
\[
i_2(G)\geqslant\frac{q^{\lfloor\frac{n^2}{2}\rfloor}}{8}=\frac{q^\frac{n^2}{2}}{8}.
\]

According to~\cite[Proposition~4.1]{King2017}, the groups $M\in\mu\setminus\calS$ form one $G$-conjugacy class and satisfy
\[
i_2(M)<3q^{\frac{n^2}{2}-\frac{n}{2}-1}\quad\text{and}\quad|\Nor_M(\langle x\rangle)|=|\Nor_G(\langle x\rangle)|.
\]
It then follows that
\begin{equation}\label{eq9}
\sum_{M\in\mu\setminus\calS}\frac{|\Nor_G(\langle x\rangle)|i_2(M)}{|\Nor_M(\langle x\rangle)|i_2(G)}
<\frac{3q^{\frac{n^2}{2}-\frac{n}{2}-1}}{i_2(G)}
<\frac{24}{q^\frac{n}{2}}.
\end{equation}

For $n\leqslant6$, we see from~\cite[Tables~8.11~and~8.27]{BHR2013} that there exist absolute constants $C_1$ and $C_2$ such that the number of conjugacy classes of $\calS$-subgroups $M$ of $G$ with order divisible by $r$ is at most $C_1$ and all such $\calS$-subgroups $M$ satisfy $i_2(M)<C_2$. This implies that for $n\leqslant6$,
\begin{align*}
\sum_{M\in\mu\cap\calS}\frac{|\Nor_G(\langle x\rangle)|i_2(M)}{|\Nor_M(\langle x\rangle)|i_2(G)}
&<\sum_{M\in\mu\cap\calS}\frac{|\Nor_G(\langle x\rangle)|i_2(M)}{i_2(G)}\\
&<\frac{C_1C_2|\Nor_G(\langle x\rangle)|}{i_2(G)}
<\frac{C_1C_2nq^{n-1}}{i_2(G)}
<\frac{8C_1C_2n}{q^\frac{n}{2}}
\leqslant\frac{2C_1C_2n^2}{q^\frac{n}{2}}.
\end{align*}
For $n\geqslant8$, we derive from Lemma~\ref{lem1} that
\[
\sum_{M\in\mu\cap\calS}i_2(M)<3\gcd(n,q+1)q^{2n+4}\leqslant3nq^{2n+4}
\]
and hence
\[
\sum_{M\in\mu\cap\calS}\frac{|\Nor_G(\langle x\rangle)|i_2(M)}{|\Nor_M(\langle x\rangle)|i_2(G)}
<\sum_{M\in\mu\cap\calS}\frac{nq^{n-1}i_2(M)}{i_2(G)}<\frac{3n^2q^{3n+3}}{i_2(G)}<\frac{2n^2}{q^\frac{n}{2}}.
\]
Therefore, we have for all even $n\geqslant4$ that
\begin{equation}\label{eq10}
\sum_{M\in\mu\cap\calS}\frac{|\Nor_G(\langle x\rangle)|i_2(M)}{|\Nor_M(\langle x\rangle)|i_2(G)}<\frac{C_0n^2}{q^\frac{n}{2}},
\end{equation}
where $C_0=\max\{2C_1C_2,2\}$.

Combining~\eqref{eq9} and~\eqref{eq10} we conclude that
\[
\sum_{M\in\mu}\frac{|\Nor_G(\langle x\rangle)|i_2(M)}{|\Nor_M(\langle x\rangle)|i_2(G)}
<\frac{24}{q^\frac{n}{2}}+\frac{C_0n^2}{q^\frac{n}{2}}
\leqslant\frac{24}{(q^n)^{\frac{1}{2}}}+\frac{C_0(\log_2(q^n))^2}{(q^n)^{\frac{1}{2}}}.
\]
This together with~\eqref{eq7} shows that~\eqref{eq8} tends to $1$ as $q^n$ tends to infinity. Then the conclusion of the lemma follows readily.
\end{proof}

\begin{lemma}\label{lem13}
Let $n\geqslant8$ be even and $q=p^f$ with prime $p$. Let $G=\PSp_n(q)$ and $r\in\ppd(p,nf)$. Suppose that $x$ is an element of order $r$ in $G$. Then the probability of a random involution $y$ in $G$ satisfying $G=\langle x,y\rangle$ tends to $1$ as $q^n$ tends to infinity.
\end{lemma}

\begin{proof}
Let $\calM$ be the set of maximal subgroups of $G$ containing $x$ and let $\mu$ be a set of $G$-conjugacy class representatives for $\calM$. Then by~\cite[Corollary~6.2]{King2017},
\begin{equation}\label{eq11}
\sum_{M\in\calM}\frac{i_2(M)}{i_2(G)}=\sum_{M\in\mu}\frac{|\Nor_G(\langle x\rangle)|i_2(M)}{|\Nor_M(\langle x\rangle)|i_2(G)}.
\end{equation}
Note from Lemmas~\ref{lem7} and~\ref{lem4} that
\[
|\Nor_G(\langle x\rangle)|=\frac{n(q^\frac{n}{2}+1)}{\gcd(2,q-1)}\leqslant n(q^\frac{n}{2}+1)<2nq^\frac{n}{2}
\]
and
\[
i_2(G)\geqslant\frac{q^{\frac{n^2}{4}+\frac{n}{2}}}{2}.
\]

By~\cite[Proposition~4.1]{King2017}, every $M\in\mu\setminus\calS$ lies in Table~\ref{tab12}, where $c(M)$ is an upper bound on the number of $G$-conjugacy classes of each type, $k(M)$ is an upper bound on $i_2(M)$, $N(M)$ is a lower bound on $|\Nor_M(\langle x\rangle)|$ and $d(M)$ is an upper bound on $c(M)k(M)/N(M)$.

\begin{table}[htbp]
\caption{The tuple $(M,c(M),k(M),N(M),d(M))$ in the proof of Lemma~\ref{lem13}}\label{tab12}
\centering
\begin{tabular}{|l|l|l|l|l|l|l|}
\hline
Class & Type of $M$ & Conditions & $c(M)$ & $k(M)$ & $N(M)$ & $d(M)$\\
\hline
$\calC_3$ & $\Sp_k(q^t).\C_t$ & $n=kt$, & $1$ & $3q^{\frac{n^2}{4t}+\frac{n}{2}}$ & $|\Nor_G(\langle x\rangle)|$ & $\frac{3q^{\frac{n^2}{8}+\frac{n}{2}}}{|\Nor_G(\langle x\rangle)|}$\\
 &  & $k$ even, &  &  &  &  \\
  &  & $t$ prime &  &  &  &  \\
$\calC_3$ & $\GU_\frac{n}{2}(q).\C_2$ & $n/2$ odd, & $1$ & $2q^{\frac{n^2}{8}+\frac{n}{4}}$ & $|\Nor_G(\langle x\rangle)|$ & $\frac{2q^{\frac{n^2}{8}+\frac{n}{4}}}{|\Nor_G(\langle x\rangle)|}$\\
 &  & $q$ odd &  &  &  &  \\
$\calC_6$ & $\C_2^{2k}.\GO_{2k}^-(2)$ & $n=2^k$, & $2$ & $2n^{2\log_2(n)+1}$ & $n+1$ & $4n^{2\log_2(n)}$\\
 &  & $q=p$ odd, &  &  &  &  \\
 &  & $r=n+1$ &  &  &  &  \\
$\calC_8$ & $\PSO_n^-(q)$ & $q$ even & $1$ & $3q^{\frac{n^2}{4}}$ & $2q^\frac{n}{2}$ & $2q^{\frac{n^2}{4}-\frac{n}{2}}$\\
\hline
\end{tabular}
\end{table}

It follows from Table~\ref{tab12} and Lemma~\ref{lem10} that
\begin{align*}
\sum_{M\in\mu\cap\calC_3}\frac{|\Nor_G(\langle x\rangle)|i_2(M)}{|\Nor_M(\langle x\rangle)|i_2(G)}
&\leqslant\sum_{t\in\pi(n)}\frac{3q^{\frac{n^2}{8}+\frac{n}{2}}}{i_2(G)}+\frac{2q^{\frac{n^2}{8}+\frac{n}{4}}}{i_2(G)}\\
&\leqslant\frac{3\log_2(n)q^{\frac{n^2}{8}+\frac{n}{2}}}{i_2(G)}+\frac{q^{\frac{n^2}{8}+\frac{n}{2}}}{i_2(G)}
\leqslant\frac{6\log_2(n)+2}{q^\frac{n}{2}}.
\end{align*}
Moreover, we deduce from Table~\ref{tab12} that
\[
\sum_{M\in\mu\cap\calC_6}\frac{|\Nor_G(\langle x\rangle)|i_2(M)}{|\Nor_M(\langle x\rangle)|i_2(G)}
<\frac{4n^{2\log_2(n)}|\Nor_G(\langle x\rangle)|}{i_2(G)}
<\frac{8n^{2\log_2(n)+1}q^\frac{n}{2}}{i_2(G)}
\leqslant\frac{16n^{2\log_2(n)+1}}{q^n},
\]
\[
\sum_{M\in\mu\cap\calC_8}\frac{|\Nor_G(\langle x\rangle)|i_2(M)}{|\Nor_M(\langle x\rangle)|i_2(G)}
\leqslant\frac{2q^{\frac{n^2}{4}-\frac{n}{2}}|\Nor_G(\langle x\rangle)|}{i_2(G)}
<\frac{4nq^\frac{n^2}{4}}{i_2(G)}
\leqslant\frac{8n}{q^\frac{n}{2}},
\]
and hence
\begin{align}\label{eq12}
\sum_{M\in\mu\setminus\calS}\frac{|\Nor_G(\langle x\rangle)|i_2(M)}{|\Nor_M(\langle x\rangle)|i_2(G)}
&<\frac{6\log_2(n)+2}{q^\frac{n}{2}}
+\frac{16n^{2\log_2(n)+1}}{q^n}
+\frac{8n}{q^\frac{n}{2}}\\\nonumber
&<\frac{5n}{q^\frac{n}{2}}+\frac{4n^{3\log_2(n)}}{q^n}+\frac{8n}{q^\frac{n}{2}}
=\frac{4n^{3\log_2(n)}}{q^n}+\frac{13n}{q^\frac{n}{2}}.
\end{align}

According to~\cite[Tables~8.49]{BHR2013}, for $n=8$ there exist absolute constants $C_1$ and $C_2$ such that the number of conjugacy classes of $\calS$-subgroups $M$ of $G$ with order divisible by $r$ is at most $C_1$ and all such $\calS$-subgroups $M$ satisfy $i_2(M)<C_2$. This implies that for $n=8$,
\begin{align*}
\sum_{M\in\mu\cap\calS}\frac{|\Nor_G(\langle x\rangle)|i_2(M)}{|\Nor_M(\langle x\rangle)|i_2(G)}
&<\sum_{M\in\mu\cap\calS}\frac{|\Nor_G(\langle x\rangle)|i_2(M)}{i_2(G)}\\
&<\frac{C_1C_2|\Nor_G(\langle x\rangle)|}{i_2(G)}
<\frac{2C_1C_2nq^\frac{n}{2}}{i_2(G)}
\leqslant\frac{4C_1C_2n}{q^n}
<\frac{C_1C_2n^3}{q^n}.
\end{align*}
For $n\geqslant10$, we derive from Lemmas~\ref{lem1} and~\ref{lem8} that
\begin{align*}
\sum_{M\in\mu\cap\calS}i_2(M)&<\gcd(2,q-1)(n+2)!+\left(n^2+\frac{21}{4}n-1\right)\gcd(2,q-1)q^{2n+4}\\
&\leqslant2(n+2)!+4n^2q^{2n+4}<2^{\frac{n^2}{4}-\frac{n}{8}+9}+4n^2q^{2n+4}
\leqslant2^9q^{\frac{n^2}{4}-\frac{n}{8}}+4n^2q^{2n+4},
\end{align*}
whence
\begin{align*}
\sum_{M\in\mu\cap\calS}\frac{|\Nor_G(\langle x\rangle)|i_2(M)}{|\Nor_M(\langle x\rangle)|i_2(G)}
&<\sum_{M\in\mu\cap\calS}\frac{2nq^\frac{n}{2}i_2(M)}{i_2(G)}\\
&<\frac{2^{10}nq^{\frac{n^2}{4}+\frac{3n}{8}}}{i_2(G)}+\frac{8n^3q^{\frac{5n}{2}+4}}{i_2(G)}
<\frac{21n^3}{q^\frac{n}{10}}+\frac{8n^3}{q^\frac{n}{10}}=\frac{29n^3}{q^\frac{n}{10}}.
\end{align*}
Therefore, we have for all even $n\geqslant8$ that
\begin{equation}\label{eq13}
\sum_{M\in\mu\cap\calS}\frac{|\Nor_G(\langle x\rangle)|i_2(M)}{|\Nor_M(\langle x\rangle)|i_2(G)}<\frac{C_0n^3}{q^\frac{n}{10}},
\end{equation}
where $C_0=\max\{C_1C_2,29\}$.

Combining~\eqref{eq12} and~\eqref{eq13} we conclude that
\[
\sum_{M\in\mu}\frac{|\Nor_G(\langle x\rangle)|i_2(M)}{|\Nor_M(\langle x\rangle)|i_2(G)}
<\frac{4n^{3\log_2(n)}}{q^n}+\frac{13n}{q^\frac{n}{2}}+\frac{C_0n^3}{q^\frac{n}{10}}.
\]
By Lemma~\ref{lem9}, for sufficiently large $q^n$ we have $n^{3\log_2(n)}<q^\frac{n}{2}$ and so
\[
\sum_{M\in\mu}\frac{|\Nor_G(\langle x\rangle)|i_2(M)}{|\Nor_M(\langle x\rangle)|i_2(G)}
<\frac{4}{q^\frac{n}{2}}+\frac{13n}{q^\frac{n}{2}}+\frac{C_0n^3}{q^\frac{n}{10}}
\leqslant\frac{4}{(q^n)^{\frac{1}{2}}}+\frac{13\log_2(q^n)}{(q^n)^{\frac{1}{2}}}+\frac{C_0(\log_2(q^n))^3}{(q^n)^{\frac{1}{10}}}.
\]
This in conjunction with~\eqref{eq11} shows that~\eqref{eq8} tends to $1$ as $q^n$ tends to infinity. Then the conclusion of the lemma follows readily.
\end{proof}

\begin{lemma}
Let $n\geqslant9$ be odd and $q=p^f$ with odd prime $p$. Let $G=\POm_n(q)$ and $r\in\ppd(p,(n-1)f)$. Suppose that $x$ is an element of order $r$ in $G$. Then the probability of a random involution $y$ in $G$ satisfying $G=\langle x,y\rangle$ tends to $1$ as $q^n$ tends to infinity.
\end{lemma}

\begin{proof}
Let $\calM$ be the set of maximal subgroups of $G$ containing $x$ and let $\mu$ be a set of $G$-conjugacy class representatives for $\calM$. Then by~\cite[Corollary~6.2]{King2017},
\begin{equation}\label{eq14}
\sum_{M\in\calM}\frac{i_2(M)}{i_2(G)}=\sum_{M\in\mu}\frac{|\Nor_G(\langle x\rangle)|i_2(M)}{|\Nor_M(\langle x\rangle)|i_2(G)}.
\end{equation}
Note from Lemmas~\ref{lem7} and~\ref{lem4} that
\[
|\Nor_G(\langle x\rangle)|=\frac{(n-1)(q^\frac{n-1}{2}+1)}{2}<nq^{\frac{n}{2}-\frac{1}{2}}
\]
and
\[
i_2(G)\geqslant\frac{q^{\frac{n^2}{4}-\frac{1}{4}}}{2}.
\]

By~\cite[Proposition~4.1]{King2017}, the groups $M\in\mu\setminus\calS$ have type $\GO_{n-1}^-(q)\times\GO_1(q)$ or $\GO_1(q)\wr\Sy_n$. Moreover, there is precisely one $G$-conjugacy class of subgroups $M$ of the first type, and they satisfy
\[
i_2(M)<6q^{\frac{n^2}{4}-\frac{n}{2}+\frac{1}{4}}\quad\text{and}\quad|\Nor_M(\langle x\rangle)|=|\Nor_G(\langle x\rangle)|;
\]
there are at most two $G$-conjugacy classes of subgroups $M$ of the second type, and they satisfy
\[
i_2(M)\leqslant2^{n-1}n!\quad\text{and}\quad|\Nor_M(\langle x\rangle)|\geqslant n.
\]
Applying Lemma~\ref{lem8} we obtain
\[
n!<3^{\frac{(n-2)^2}{4}-\frac{5(n-2)}{8}+4}=3^{\frac{n^2}{4}-\frac{13n}{8}+\frac{25}{4}}\leqslant3^\frac{25}{4}q^{\frac{n^2}{4}-\frac{13n}{8}}
\]
and so
\[
2^n|\Nor_G(\langle x\rangle)|n!<2^9q^{n-9}\cdot nq^{\frac{n}{2}-\frac{1}{2}}\cdot3^\frac{25}{4}q^{\frac{n^2}{4}-\frac{13n}{8}}
=2^93^\frac{25}{4}nq^{\frac{n^2}{4}-\frac{n}{8}-\frac{19}{2}}.
\]
Then it follows that
\begin{align}\label{eq15}
\sum_{M\in\mu\setminus\calS}\frac{|\Nor_G(\langle x\rangle)|i_2(M)}{|\Nor_M(\langle x\rangle)|i_2(G)}
&<\frac{6q^{\frac{n^2}{4}-\frac{n}{2}+\frac{1}{4}}}{i_2(G)}+\frac{2|\Nor_G(\langle x\rangle)|2^{n-1}n!}{ni_2(G)}\\\nonumber
&<\frac{6q^{\frac{n^2}{4}-\frac{n}{2}+\frac{1}{4}}}{i_2(G)}+\frac{2^93^\frac{25}{4}q^{\frac{n^2}{4}-\frac{n}{8}-\frac{19}{2}}}{i_2(G)}
<\frac{1}{q^\frac{n}{8}}+\frac{38}{q^\frac{n}{8}}=\frac{39}{q^\frac{n}{8}}.
\end{align}

For $n\leqslant11$, we see from~\cite[Tables~8.59~and~8.75]{BHR2013} that there exist absolute constants $C_1$ and $C_2$ such that the number of conjugacy classes of $\calS$-subgroups $M$ of $G$ with order divisible by $r$ is at most $C_1$ and all such $\calS$-subgroups $M$ satisfy $i_2(M)<C_2$. This implies that for $n\leqslant11$,
\begin{align*}
\sum_{M\in\mu\cap\calS}\frac{|\Nor_G(\langle x\rangle)|i_2(M)}{|\Nor_M(\langle x\rangle)|i_2(G)}
&<\sum_{M\in\mu\cap\calS}\frac{|\Nor_G(\langle x\rangle)|i_2(M)}{i_2(G)}\\
&<\frac{C_1C_2|\Nor_G(\langle x\rangle)|}{i_2(G)}
<\frac{C_1C_2nq^{\frac{n}{2}-\frac{1}{2}}}{i_2(G)}
<\frac{C_1C_2n}{q^\frac{n}{8}}
<\frac{C_1C_2n^3}{q^\frac{n}{8}}.
\end{align*}
For $n\geqslant13$, we derive from Lemmas~\ref{lem1} and~\ref{lem8} that
\begin{align*}
\sum_{M\in\mu\cap\calS}i_2(M)&<2(n+2)!+(2n^2+12n+8)q^{2n+4}\\
&<2\cdot3^4\cdot3^{\frac{n^2}{4}-\frac{5n}{8}}+3n^2q^{2n+4}
<n^2q^{\frac{n^2}{4}-\frac{5n}{8}}+3n^2q^{2n+4}
\end{align*}
and hence
\begin{align*}
\sum_{M\in\mu\cap\calS}\frac{|\Nor_G(\langle x\rangle)|i_2(M)}{|\Nor_M(\langle x\rangle)|i_2(G)}
&<\sum_{M\in\mu\cap\calS}\frac{nq^{\frac{n}{2}-\frac{1}{2}}i_2(M)}{i_2(G)}\\
&<\frac{n^3q^{\frac{n^2}{4}-\frac{n}{8}-\frac{1}{2}}}{i_2(G)}+\frac{3n^3q^{\frac{5n}{2}+\frac{7}{2}}}{i_2(G)}
<\frac{2n^3}{q^\frac{n}{8}}+\frac{n^3}{q^\frac{n}{8}}=\frac{3n^3}{q^\frac{n}{8}}.
\end{align*}
Therefore, we have for all odd $n\geqslant9$ that
\begin{equation}\label{eq16}
\sum_{M\in\mu\cap\calS}\frac{|\Nor_G(\langle x\rangle)|i_2(M)}{|\Nor_M(\langle x\rangle)|i_2(G)}<\frac{C_0n^3}{q^\frac{n}{8}},
\end{equation}
where $C_0=\max\{C_1C_2,3\}$.

Combining~\eqref{eq15} and~\eqref{eq16} we conclude that
\[
\sum_{M\in\mu}\frac{|\Nor_G(\langle x\rangle)|i_2(M)}{|\Nor_M(\langle x\rangle)|i_2(G)}
<\frac{39}{q^\frac{n}{8}}+\frac{C_0n^3}{(q^n)^{\frac{1}{8}}}
\leqslant\frac{39}{q^\frac{n}{8}}+\frac{C_0(\log_3(q^n))^3}{(q^n)^{\frac{1}{8}}}.
\]
This together with~\eqref{eq14} shows that~\eqref{eq8} tends to $1$ as $q^n$ tends to infinity. Then the conclusion of the lemma follows readily.
\end{proof}

\begin{lemma}\label{lem14}
Let $n\geqslant10$ be even and $q=p^f$ with prime $p$. Let $G=\POm_n^+(q)$ and $r\in\ppd(p,(n-2)f)$. Suppose that $x$ is an element of order $r$ in $G$. Then the probability of a random involution $y$ in $G$ satisfying $G=\langle x,y\rangle$ tends to $1$ as $q^n$ tends to infinity.
\end{lemma}

\begin{proof}
Let $\calM$ be the set of maximal subgroups of $G$ containing $x$ and let $\mu$ be a set of $G$-conjugacy class representatives for $\calM$. Then by~\cite[Corollary~6.2]{King2017},
\begin{equation}\label{eq19}
\sum_{M\in\calM}\frac{i_2(M)}{i_2(G)}=\sum_{M\in\mu}\frac{|\Nor_G(\langle x\rangle)|i_2(M)}{|\Nor_M(\langle x\rangle)|i_2(G)}.
\end{equation}
Note from Lemmas~\ref{lem7} and~\ref{lem4} that
\[
|\Nor_G(\langle x\rangle)|=\frac{2(n-2)(q^{\frac{n}{2}-1}+1)(q+1)}{\gcd(2,q-1)^2|\PSO_n^+(q)/G|}\leqslant(n-2)(q^{\frac{n}{2}-1}+1)(q+1)<2nq^\frac{n}{2}
\]
and
\[
i_2(G)\geqslant\frac{q^{\frac{n^2}{4}-1}}{8}.
\]

By~\cite[Proposition~4.1]{King2017}, every $M\in\mu\setminus\calS$ lies in Table~\ref{tab13}, where $c(M)$ is an upper bound on the number of $G$-conjugacy classes of each type, $k(M)$ is an upper bound on $i_2(M)$, $N(M)$ is a lower bound on $|\Nor_M(\langle x\rangle)|$ and $d(M)$ is an upper bound on $c(M)k(M)/N(M)$.

\begin{table}[htbp]
\caption{The tuple $(M,c(M),k(M),N(M),d(M))$ in the proof of Lemma~\ref{lem14}}\label{tab13}
\centering
\begin{tabular}{|l|l|l|l|l|l|l|}
\hline
Class & Type of $M$ & Conditions & $c(M)$ & $k(M)$ & $N(M)$ & $d(M)$\\
\hline
$\calC_1$ & $\GO_{n-2}^-(q)\times\GO_2^-(q)$ &  & $1$ & $6q^{\frac{n^2}{4}-n+2}$ & $|\Nor_G(\langle x\rangle)|$ & $\frac{6q^{\frac{n^2}{4}-n+2}}{|\Nor_G(\langle x\rangle)|}$\\
$\calC_1$ & $\GO_{n-1}(q)\times\GO_1(q)$ &  & $2$ & $6q^{\frac{n^2}{4}-\frac{n}{2}}$ & $nq^{\frac{n}{2}-3}$ & $\frac{12q^{\frac{n^2}{4}-n+3}}{n}$\\
$\calC_2$ & $\GO_1(q)\wr\Sy_n$ & $q=p$ odd, & $4$ & $2^{n-1}n!$ & $n-1$ & $\frac{2^{n+1}n!}{n-1}$\\
 &  & $r=n-1$ &  &  &  &  \\
$\calC_3$ & $\GO_\frac{n}{2}(q^2).\C_2$ & $n/2$ odd, & $2$ & $6q^{\frac{n^2}{8}-\frac{3}{2}}$ & $\frac{nq^{\frac{n}{2}-2}}{4}$ & $\frac{48q^{\frac{n^2}{8}-\frac{n}{2}+\frac{1}{2}}}{n}$\\
 &  & $q$ odd &  &  &  &  \\
$\calC_3$ & $\GU_\frac{n}{2}(q).\C_2$ & $n/2$ even & $2$ & $5q^{\frac{n^2}{8}+\frac{n}{4}}$ & $|\Nor_G(\langle x\rangle)|$ & $\frac{10q^{\frac{n^2}{8}+\frac{n}{4}}}{|\Nor_G(\langle x\rangle)|}$\\
$\calC_6$ & $\C_2^{2k}.\GO_{2k}^+(2)$ & $n=2^k$, & $8$ & $n^{2\log_2(n)+1}$ & $n-1$ & $n^{2\log_2(n)+1}$\\
 &  & $q=p$ odd, &  &  &  &  \\
 &  & $r=n-1$ &  &  &  &  \\
\hline
\end{tabular}
\end{table}

For odd $q$, Lemma~\ref{lem8} gives that
\[
n!<3^{\frac{(n-2)^2}{4}-\frac{5(n-2)}{8}+4}=3^{\frac{n^2}{4}-\frac{13n}{8}+\frac{25}{4}}\leqslant3^\frac{25}{4}q^{\frac{n^2}{4}-\frac{13n}{8}}
\]
and so
\[
\frac{2^{n+1}n!|\Nor_G(\langle x\rangle)|}{n-1}<\frac{2^{11}q^{n-10}\cdot3^\frac{25}{4}q^{\frac{n^2}{4}-\frac{13n}{8}}\cdot2nq^\frac{n}{2}}{\frac{9n}{10}}
=2^{13}3^\frac{17}{4}5q^{\frac{n^2}{4}-\frac{n}{8}-10}.
\]
Then it follows from Table~\ref{tab13} that
\[
\sum_{M\in\mu\cap\calC_2}\frac{|\Nor_G(\langle x\rangle)|i_2(M)}{|\Nor_M(\langle x\rangle)|i_2(G)}
\leqslant\frac{2^{n+1}n!|\Nor_G(\langle x\rangle)|}{(n-1)i_2(G)}
<\frac{2^{13}3^\frac{17}{4}5q^{\frac{n^2}{4}-\frac{n}{8}-10}}{i_2(G)}
<\frac{1349}{q^\frac{n}{10}}.
\]
Moreover, we deduce from Table~\ref{tab13} that
\begin{align*}
\sum_{M\in\mu\cap\calC_1}\frac{|\Nor_G(\langle x\rangle)|i_2(M)}{|\Nor_M(\langle x\rangle)|i_2(G)}
&\leqslant\frac{6q^{\frac{n^2}{4}-n+2}}{i_2(G)}+\frac{12q^{\frac{n^2}{4}-n+3}|\Nor_G(\langle x\rangle)|}{ni_2(G)}\\
&<\frac{6q^{\frac{n^2}{4}-n+2}}{i_2(G)}+\frac{24q^{\frac{n^2}{4}-\frac{n}{2}+3}}{i_2(G)}
<\frac{1}{q^\frac{n}{10}}+\frac{192}{q^\frac{n}{10}}=\frac{193}{q^\frac{n}{10}},
\end{align*}
\begin{align*}
\sum_{M\in\mu\cap\calC_3}\frac{|\Nor_G(\langle x\rangle)|i_2(M)}{|\Nor_M(\langle x\rangle)|i_2(G)}
&\leqslant\frac{48q^{\frac{n^2}{8}-\frac{n}{2}+\frac{1}{2}}|\Nor_G(\langle x\rangle)|}{ni_2(G)}+\frac{10q^{\frac{n^2}{8}+\frac{n}{4}}}{i_2(G)}\\
&<\frac{96q^{\frac{n^2}{8}+\frac{1}{2}}}{i_2(G)}+\frac{10q^{\frac{n^2}{8}+\frac{n}{4}}}{i_2(G)}
<\frac{16}{q^\frac{n}{10}}+\frac{1}{q^\frac{n}{10}}=\frac{17}{q^\frac{n}{10}},
\end{align*}
\[
\sum_{M\in\mu\cap\calC_6}\frac{|\Nor_G(\langle x\rangle)|i_2(M)}{|\Nor_M(\langle x\rangle)|i_2(G)}
\leqslant\frac{n^{2\log_2(n)+1}|\Nor_G(\langle x\rangle)|}{i_2(G)}
<\frac{2n^{2\log_2(n)+2}q^\frac{n}{2}}{i_2(G)}
<\frac{n^{2\log_2(n)}}{q^\frac{n}{2}},
\]
and hence
\begin{align}\label{eq20}
\sum_{M\in\mu\setminus\calS}\frac{|\Nor_G(\langle x\rangle)|i_2(M)}{|\Nor_M(\langle x\rangle)|i_2(G)}
<\frac{1349}{q^\frac{n}{10}}+\frac{193}{q^\frac{n}{10}}+\frac{17}{q^\frac{n}{10}}+\frac{n^{2\log_2(n)}}{q^\frac{n}{2}}
=\frac{1559}{q^\frac{n}{10}}+\frac{n^{2\log_2(n)}}{q^\frac{n}{2}}.
\end{align}

For $n=10$ we see from~\cite[Table~8.67]{BHR2013} that there exist absolute constants $C_1$ and $C_2$ such that the number of conjugacy classes of $\calS$-subgroups $M$ of $G$ with order divisible by $r$ is at most $C_1$ and all such $\calS$-subgroups $M$ satisfy $i_2(M)<C_2$. This implies that for $n=10$,
\begin{align*}
\sum_{M\in\mu\cap\calS}\frac{|\Nor_G(\langle x\rangle)|i_2(M)}{|\Nor_M(\langle x\rangle)|i_2(G)}
&<\sum_{M\in\mu\cap\calS}\frac{|\Nor_G(\langle x\rangle)|i_2(M)}{i_2(G)}\\
&<\frac{C_1C_2|\Nor_G(\langle x\rangle)|}{i_2(G)}
<\frac{2C_1C_2nq^\frac{n}{2}}{i_2(G)}
<\frac{C_1C_2n}{q^\frac{n}{12}}
<\frac{C_1C_2n^3}{q^\frac{n}{12}}.
\end{align*}
Applying Lemma~\ref{lem8} we obtain
\[
n!<2^{\frac{(n-2)^2}{4}-\frac{n-2}{8}+8}=2^{\frac{n^2}{4}-\frac{9n}{8}+\frac{37}{4}}\leqslant2^\frac{37}{4}q^{\frac{n^2}{4}-\frac{9n}{8}}
\]
and thus
\[
(n+2)!<2^\frac{37}{4}q^{\frac{n^2}{4}-\frac{9n}{8}}(n+1)(n+2)<2^{10}n^2q^{\frac{n^2}{4}-\frac{9n}{8}}.
\]
Then for $n\geqslant12$ we derive from Lemma~\ref{lem1} that
\[
\sum_{M\in\mu\cap\calS}i_2(M)<8(n+2)!+(2n+72)q^{2n+4}<2^{13}n^2q^{\frac{n^2}{4}-\frac{9n}{8}}+n^2q^{2n+4},
\]
whence
\begin{align*}
\sum_{M\in\mu\cap\calS}\frac{|\Nor_G(\langle x\rangle)|i_2(M)}{|\Nor_M(\langle x\rangle)|i_2(G)}
<\sum_{M\in\mu\cap\calS}\frac{2nq^\frac{n}{2}i_2(M)}{i_2(G)}
&<\frac{2^{14}n^3q^{\frac{n^2}{4}-\frac{5n}{8}}}{i_2(G)}+\frac{2n^3q^{\frac{5n}{2}+4}}{i_2(G)}\\
&<\frac{2897n^3}{q^\frac{n}{12}}+\frac{16n^3}{q^\frac{n}{12}}=\frac{2913n^3}{q^\frac{n}{12}}.
\end{align*}
Therefore, we have for all even $n\geqslant10$ that
\begin{equation}\label{eq21}
\sum_{M\in\mu\cap\calS}\frac{|\Nor_G(\langle x\rangle)|i_2(M)}{|\Nor_M(\langle x\rangle)|i_2(G)}<\frac{C_0n^3}{q^\frac{n}{12}},
\end{equation}
where $C_0=\max\{C_1C_2,2913\}$.

Combining~\eqref{eq20} and~\eqref{eq21} we conclude that
\[
\sum_{M\in\mu}\frac{|\Nor_G(\langle x\rangle)|i_2(M)}{|\Nor_M(\langle x\rangle)|i_2(G)}
<\frac{1559}{q^\frac{n}{10}}+\frac{n^{2\log_2(n)}}{q^\frac{n}{2}}+\frac{C_0n^3}{q^\frac{n}{12}}.
\]
By Lemma~\ref{lem9}, for sufficiently large $q^n$ we have $n^{2\log_2(n)}<q^\frac{n}{4}$ and so
\[
\sum_{M\in\mu}\frac{|\Nor_G(\langle x\rangle)|i_2(M)}{|\Nor_M(\langle x\rangle)|i_2(G)}
<\frac{1559}{q^\frac{n}{10}}+\frac{1}{q^\frac{n}{4}}+\frac{C_0n^3}{q^\frac{n}{12}}
\leqslant\frac{1559}{(q^n)^{\frac{1}{10}}}+\frac{1}{(q^n)^{\frac{1}{4}}}+\frac{C_0(\log_2(q^n))^3}{(q^n)^{\frac{1}{12}}}.
\]
This in conjunction with~\eqref{eq19} shows that~\eqref{eq8} tends to $1$ as $q^n$ tends to infinity. Then the conclusion of the lemma follows readily.
\end{proof}

\begin{lemma}\label{lem12}
Let $n\geqslant8$ be even and $q=p^f$ with prime $p$. Let $G=\POm_n^-(q)$ and $r\in\ppd(p,nf)$. Suppose that $x$ is an element of order $r$ in $G$. Then the probability of a random involution $y$ in $G$ satisfying $G=\langle x,y\rangle$ tends to $1$ as $q^n$ tends to infinity.
\end{lemma}

\begin{proof}
Let $\calM$ be the set of maximal subgroups of $G$ containing $x$ and let $\mu$ be a set of $G$-conjugacy class representatives for $\calM$. Then by~\cite[Corollary~6.2]{King2017},
\begin{equation}\label{eq22}
\sum_{M\in\calM}\frac{i_2(M)}{i_2(G)}=\sum_{M\in\mu}\frac{|\Nor_G(\langle x\rangle)|i_2(M)}{|\Nor_M(\langle x\rangle)|i_2(G)}.
\end{equation}
Note from Lemmas~\ref{lem7} and~\ref{lem4} that
\[
|\Nor_G(\langle x\rangle)|=\frac{n(q^\frac{n}{2}+1)}{\gcd(2,q-1)^2|\PSO_n^-(q)/G|}
\leqslant\frac{n(q^\frac{n}{2}+1)}{2}<nq^\frac{n}{2}
\]
and
\[
i_2(G)\geqslant\frac{q^{\frac{n^2}{4}-1}}{8}.
\]

By~\cite[Proposition~4.1]{King2017}, the groups $M\in\mu\setminus\calS$ have type $\GO_k^-(q^t).\C_t$ with $n=kt\geqslant4t$ and $t$ prime or $\GU_\frac{n}{2}(q).\C_2$ with $n/2$ odd. Moreover, there is exactly one $G$-conjugacy class of subgroups $M$ of the first type for each $t$, and they satisfy
\[
i_2(M)<3q^\frac{n^2}{4t}\quad\text{and}\quad|\Nor_M(\langle x\rangle)|=|\Nor_G(\langle x\rangle)|;
\]
there is precisely one $G$-conjugacy class of subgroups $M$ of the second type, and they satisfy
\[
i_2(M)<5q^{\frac{n^2}{8}+\frac{n}{4}}\quad\text{and}\quad|\Nor_M(\langle x\rangle)|=|\Nor_G(\langle x\rangle)|.
\]
It then follows from Lemma~\ref{lem10} that
\begin{align}\label{eq23}
\sum_{M\in\mu\setminus\calS}\frac{|\Nor_G(\langle x\rangle)|i_2(M)}{|\Nor_M(\langle x\rangle)|i_2(G)}
&<\sum_{t\in\pi(n)}\frac{3q^\frac{n^2}{4t}}{i_2(G)}+\frac{5q^{\frac{n^2}{8}+\frac{n}{4}}}{i_2(G)}\\\nonumber
&\leqslant\frac{3\log_2(n)q^\frac{n^2}{8}}{i_2(G)}+\frac{5q^{\frac{n^2}{8}+\frac{n}{4}}}{i_2(G)}
\leqslant\frac{3\log_2(n)}{q^\frac{n}{2}}+\frac{20}{q^\frac{n}{2}}<\frac{4n}{q^\frac{n}{2}}.
\end{align}

For $n\leqslant10$ we see from~\cite[Tables~8.53~and~8.69]{BHR2013} that there exist absolute constants $C_1$ and $C_2$ such that the number of conjugacy classes of $\calS$-subgroups $M$ of $G$ with order divisible by $r$ is at most $C_1$ and all such $\calS$-subgroups $M$ satisfy $i_2(M)<C_2$. This implies that for $n\leqslant10$,
\begin{align*}
\sum_{M\in\mu\cap\calS}\frac{|\Nor_G(\langle x\rangle)|i_2(M)}{|\Nor_M(\langle x\rangle)|i_2(G)}
&<\sum_{M\in\mu\cap\calS}\frac{|\Nor_G(\langle x\rangle)|i_2(M)}{i_2(G)}\\
&<\frac{C_1C_2|\Nor_G(\langle x\rangle)|}{i_2(G)}
<\frac{C_1C_2nq^\frac{n}{2}}{i_2(G)}
<\frac{C_1C_2n}{q^\frac{n}{12}}
<\frac{C_1C_2n^3}{q^\frac{n}{12}}.
\end{align*}
If $n\geqslant12$, then applying Lemma~\ref{lem8} we obtain
\[
n!<2^{\frac{(n-2)^2}{4}-\frac{n-2}{8}+8}=2^{\frac{n^2}{4}-\frac{9n}{8}+\frac{37}{4}}\leqslant2^\frac{37}{4}q^{\frac{n^2}{4}-\frac{9n}{8}}
\]
and so
\[
(n+2)!<2^\frac{37}{4}q^{\frac{n^2}{4}-\frac{9n}{8}}(n+1)(n+2)<2^{10}n^2q^{\frac{n^2}{4}-\frac{9n}{8}}.
\]
Thus for $n\geqslant12$ we derive from Lemma~\ref{lem1} that
\[
\sum_{M\in\mu\cap\calS}i_2(M)<8(n+2)!+8\left(n^2+\frac{21}{4}n-1\right)q^{2n+4}<2^{13}n^2q^{\frac{n^2}{4}-\frac{9n}{8}}+16n^2q^{2n+4},
\]
whence
\begin{align*}
\sum_{M\in\mu\cap\calS}\frac{|\Nor_G(\langle x\rangle)|i_2(M)}{|\Nor_M(\langle x\rangle)|i_2(G)}
<\sum_{M\in\mu\cap\calS}\frac{nq^\frac{n}{2}i_2(M)}{i_2(G)}
&<\frac{2^{13}n^3q^{\frac{n^2}{4}-\frac{5n}{8}}}{i_2(G)}+\frac{16n^3q^{\frac{5n}{2}+4}}{i_2(G)}\\
&<\frac{1449n^3}{q^\frac{n}{12}}+\frac{128n^3}{q^\frac{n}{12}}=\frac{1577n^3}{q^\frac{n}{12}}.
\end{align*}
Therefore, we have for all even $n\geqslant8$ that
\begin{equation}\label{eq24}
\sum_{M\in\mu\cap\calS}\frac{|\Nor_G(\langle x\rangle)|i_2(M)}{|\Nor_M(\langle x\rangle)|i_2(G)}<\frac{C_0n^3}{q^\frac{n}{12}},
\end{equation}
where $C_0=\max\{C_1C_2,1577\}$.

Combining~\eqref{eq23} and~\eqref{eq24} we conclude that
\[
\sum_{M\in\mu}\frac{|\Nor_G(\langle x\rangle)|i_2(M)}{|\Nor_M(\langle x\rangle)|i_2(G)}
<\frac{3n}{q^\frac{n}{2}}+\frac{C_0n^3}{q^\frac{n}{12}}
\leqslant\frac{3\log_2(q^n)}{(q^n)^{\frac{1}{2}}}+\frac{C_0(\log_2(q^n))^3}{(q^n)^{\frac{1}{12}}}.
\]
This in conjunction with~\eqref{eq22} shows that~\eqref{eq8} tends to $1$ as $q^n$ tends to infinity. Then the conclusion of the lemma follows readily.
\end{proof}

\section{Probability for GRR}

To complete the proof of Theorem~\ref{thm1} we need the following two technical lemmas.

\begin{lemma}\label{lem5}
Let $G$ be a finite simple classical group with natural module $V$. Suppose that $\alpha$ is an involution in $\Aut(G)\setminus\PGL(V)$. Then $i_2(\Cen_G(\alpha))<\ell(G)$ with $\ell(G)$ given in Table~$\ref{tab3}$, where $\ell(G)=0$ means that $\alpha$ does not exist for such $G$.
\end{lemma}

\begin{table}[htbp]
\caption{The pair $(G,\ell(G))$ in Lemma~\ref{lem5}}\label{tab3}
\centering
\begin{tabular}{|l|l|l|}
\hline
$G$ & Conditions & $\ell(G)$\\
\hline
$\PSL_n(q)$ & $n\geqslant2$, $(n,q)\neq(2,2)$ or $(2,3)$ & $3q^{\frac{n^2}{4}+\frac{n}{2}}$\\
$\PSU_n(q)$ & $n\geqslant3$, $(n,q)\neq(3,2)$ & $3q^{\frac{n^2}{4}+\frac{n}{4}}$\\
$\PSp_n(q)$ & $n\geqslant4$ even, $(n,q)\neq(4,2)$ & $3q^{\frac{n^2}{8}+\frac{n}{4}}$\\
$\POm_n(q)$ & $n\geqslant7$ odd, $q$ odd & $3q^{\frac{n^2-1}{8}}$\\
$\POm_n^+(q)$ & $n\geqslant8$ even & $3q^{\frac{n^2}{8}}$\\
$\POm_n^-(q)$ & $n\geqslant8$ even & $0$\\
\hline
\end{tabular}
\end{table}

\begin{proof}
We derive the conclusion of the lemma for each row of Table~\ref{tab3}.

\underline{Row~1.} $G=\PSL_n(q)$ with $n\geqslant2$ and $(n,q)\neq(2,2)$ or $(2,3)$. For $(n,q)=(2,4)$ or $(2,9)$ computation in \magma~\cite{magma} directly verifies the conclusion of the lemma. Thus we may assume $(n,q)\neq(2,4)$ or $(2,9)$. If $\alpha$ is a field automorphism or a graph-field automorphism of $G$, then we deduce from~\cite[Proposition~4.4]{LS1996} and~\cite[\S4.5~and~\S4.8]{KL1990} that $\Cen_G(\alpha)$ is an almost simple group with socle $\PSL_n(q^{\frac{1}{2}})$ or $\PSU_n(q^{\frac{1}{2}})$. In this case Lemma~\ref{lem2} implies that $i_2(\Cen_G(\alpha))<3q^{\frac{n^2}{4}+\frac{n}{4}-\frac{1}{2}}$. If $\alpha$ is a graph automorphism of $G$, then by~\cite[Proposition~4.4]{LS1996}, either $n$ is odd and $\Cen_G(\alpha)=\PSO_n(q)$, or $n$ is even and $\Cen_G(\alpha)\leqslant\PSp_n(q)$ or $\PSO_n^\pm(q)$. In this case we deduce from Lemma~\ref{lem2} that $i_2(\Cen_G(\alpha))<3q^{\frac{n^2}{4}+\frac{n}{2}}$. In both cases we have $i_2(\Cen_G(\alpha))<3q^{\frac{n^2}{4}+\frac{n}{2}}$, as the lemma asserts.

\underline{Row~2.} $G=\PSU_n(q)$ with $n\geqslant3$ and $(n,q)\neq(3,2)$. For $(n,q)=(3,4)$ computation in \magma~\cite{magma} directly verifies the conclusion of the lemma. Thus we may assume that $(n,q)\neq(3,4)$. According to~\cite[Proposition~3.3.18]{BG2016}, $\alpha$ is a graph automorphism of $G$. Then by~\cite[Proposition~4.4]{LS1996}, either $n$ is odd and $\Cen_G(\alpha)=\PSO_n(q)$, or $n$ is even and $\Cen_G(\alpha)\leqslant\PSp_n(q)$ or $\PSO_n^\pm(q)$. Hence it follows from Lemma~\ref{lem2} that $i_2(\Cen_G(\alpha))<3q^{\frac{n^2}{4}+\frac{n}{2}}$, as the lemma asserts.

\underline{Row~3.} $G=\PSp_n(q)$ with $n\geqslant4$ even and $(n,q)\neq(4,2)$. For $(n,q)=(4,4)$ computation in \magma~\cite{magma} directly verifies the conclusion of the lemma. Thus we may assume $(n,q)\neq(4,4)$. Then by~\cite[Proposition~4.4]{LS1996} and~\cite[\S4.5]{KL1990}, $\Cen_G(\alpha)$ is an almost simple group with socle $\PSp_n(q^{\frac{1}{2}})$. Hence we deduce from Lemma~\ref{lem2} that $i_2(\Cen_G(\alpha))<3q^{\frac{n^2}{8}+\frac{n}{4}}$.

\underline{Row~4.} $G=\POm_n(q)$ with $n\geqslant7$ odd and $q$ odd. We derive from~\cite[Proposition~4.4]{LS1996} and~\cite[\S4.5]{KL1990} that $\Cen_G(\alpha)$ is an almost simple group with socle $\POm_n(q^{\frac{1}{2}})$. Then Lemma~\ref{lem2} implies that $i_2(\Cen_G(\alpha))<3q^{\frac{n^2-1}{8}}$.

\underline{Row~5.} $G=\POm_n^+(q)$ with $n\geqslant8$ even. By~\cite[Proposition~4.4]{LS1996} and~\cite[\S4.5]{KL1990}, $\Cen_G(\alpha)$ is an almost simple group with socle $\POm_n^+(q^{\frac{1}{2}})$ or $\POm_n^-(q^{\frac{1}{2}})$. It then follows from Lemma~\ref{lem2} that $i_2(\Cen_G(\alpha))<3q^{\frac{n^2}{8}}$.

\underline{Row~6.} $G=\POm_n^-(q)$ with $n\geqslant8$ even. By~\cite[Proposition~3.5.25]{BG2016} there is no such $\alpha$. This completes the proof.
\end{proof}

\begin{lemma}\label{lem3}
Let $G$ be a finite simple classical group with natural module $V$ and $r\in\ppd(p,ef)$, where $G$ and $e$ are given in Table~$\ref{tab6}$ with $q=p^f$, $p$ prime and $\ppd(p,ef)\neq\emptyset$. Suppose that $x\in G$ has order $r$ and $\alpha$ is an involution in $\Aut(G)\cap\PGL(V)$ such that $x^\alpha=x^{-1}$. Then $i_2(\Cen_G(\alpha))<m(G)i_2(G)$ with $m(G)$ given in Table~$\ref{tab6}$, where $C$ is an absolute constant and $m(G)=0$ means that $\alpha$ does not exist for such $G$.
\end{lemma}

\begin{table}[htbp]
\caption{The triple $(G,e,m(G))$ in Lemma~\ref{lem3}}\label{tab6}
\centering
\begin{tabular}{|l|l|l|l|}
\hline
$G$ & Conditions & $e$ & $m(G)$\\
\hline
$\PSL_n(q)$ & $n\geqslant2$, $(n,q)\neq(2,2)$ or $(2,3)$ & $n$ & $Cq^{-\frac{n^2}{10}}$\\
$\PSU_n(q)$ & $n\geqslant3$ odd, $(n,q)\neq(3,2)$ & $2n$ & $0$\\
$\PSU_n(q)$ & $n\geqslant4$ even & $2(n-1)$ & $0$\\
$\PSp_n(q)$ & $n\geqslant4$ even, $(n,q)\neq(4,2)$ & $n$ & $Cq^{-\frac{n^2}{20}-\frac{n}{10}}$\\
$\POm_n(q)$ & $n\geqslant7$ odd, $q$ odd & $n-1$ & $Cq^{-\frac{n^2}{10}+\frac{1}{2}}$\\
$\POm_n^+(q)$ & $n\geqslant8$ even & $n-2$ & $Cq^{-\frac{n^2}{20}+\frac{n}{10}}$\\
$\POm_n^-(q)$ & $n\geqslant8$ even & $n$ & $Cq^{-\frac{n^2}{20}+\frac{n}{10}}$\\
\hline
\end{tabular}
\end{table}

\begin{proof}
We derive the conclusion of the lemma for each row of Table~\ref{tab6}.

\underline{Row~1.} $G=\PSL_n(q)$ with $n\geqslant2$ and $(n,q)\neq(2,2)$ or $(2,3)$. Here we have $\Aut(G)\cap\PGL(V)=\PGL_n(q)$. Let $N=\C_{(q^n-1)/(q-1)}\rtimes\C_n$ be the normalizer of the Singer cycle of $\PGL_n(q)$ containing $x$. Since $N$ contains a Sylow $r$-subgroup of $\PGL_n(q)$, each element of order $r$ in $G$ is contained in a $\PGL_n(q)$-conjugate of $N$. Thus we may assume without loss of generality that $N$ consists of elements induced by some $\gamma\in\GL_n(q)$ such that
\[
\gamma:v\mapsto\omega v^{q^{n/t}},\quad v\in\bbF_{q^n},
\]
where $\omega\in\bbF_{q^n}^\times$ with $\bbF_{q^n}$ viewed as a vector space of dimension $n$ over $\bbF_q$ and $t$ is a divisor of $n$. Note that $N\leqslant\Nor_{\PGL_n(q)}(\langle x\rangle)$. By Lemma~\ref{lem7},
\[
|\Nor_G(\langle x\rangle)||\PGL_n(q)/G|=\frac{n(q^n-1)|\PGL_n(q)/G|}{(q-1)\gcd(n,q-1)}=\frac{n(q^n-1)}{q-1}=|N|.
\]
It follows that
\[
|N|\leqslant|\Nor_{\PGL_n(q)}(\langle x\rangle)|\leqslant|\Nor_G(\langle x\rangle)||\PGL_n(q)/G|=|N|
\]
and so $\Nor_{\PGL_n(q)}(\langle x\rangle)=N$. Consequently, $\alpha\in N$ as $\alpha\in\Nor_{\PGL_n(q)}(\langle x\rangle)$. Now $\alpha$ is an involution in $N=\C_{(q^n-1)/(q-1)}\rtimes\C_n$ such that $x^\alpha=x^{-1}$. We conclude that $n$ is even and $\alpha$ is induced by some $\beta\in\GL_n(q)$ such that
\[
\beta:v\mapsto\omega v^{q^{n/2}},\quad v\in\bbF_{q^n},
\]
where $\omega\in\bbF_{q^n}^\times$. For any $\lambda\in\bbF_q$, the equation
\[
\omega v^{q^{n/2}}=\lambda v
\]
has at most $q^{n/2}$ solutions $v\in\bbF_{q^n}$. This shows that $\beta$ has no eigenspace of dimension larger than $n/2$ over $\bbF_q$. Then we see from the proof of~\cite[Proposition~4.1]{LS1996} that $\beta$ is conjugate to
\[
\begin{pmatrix}
I_{\frac{n}{2}}&0\\
0&-I_{\frac{n}{2}}
\end{pmatrix}
\quad\text{or}\quad
\begin{pmatrix}
0&I_{\frac{n}{2}}\\
\mu I_{\frac{n}{2}}&0
\end{pmatrix}
\]
(where $\mu$ is a non-square in $\bbF_q$) if $q$ is odd, and conjugate to
\[
\begin{pmatrix}
I_{\frac{n}{2}}&0\\
I_{\frac{n}{2}}&I_{\frac{n}{2}}
\end{pmatrix}
\]
if $q$ is even; $\Cen_G(\alpha)$ is contained in a maximal subgroup $M$ of $G$ such that $M$ either stabilizes a subspace of dimension $n/2$ over $\bbF_q$, or is of type $\GL_{\frac{n}{2}}(q)\wr\Sy_2$ or $\GL_{\frac{n}{2}}(q^2).\C_2$. If $M$ stabilizes a subspace of dimension $n/2$ over $\bbF_q$ or is of type $\GL_{\frac{n}{2}}(q)\wr\Sy_2$, then
\[
|G{:}M|\geqslant\prod_{i=1}^{\frac{n}{2}}\frac{q^{\frac{n}{2}+i}-1}{q^i-1}>\prod_{i=1}^{\frac{n}{2}}q^{\frac{n}{2}}=q^{\frac{n^2}{4}}.
\]
If $M$ is of type $\GL_{\frac{n}{2}}(q^2).\C_2$, then by~\cite[Proposition~4.3.6]{KL1990},
\[
|G{:}M|=\frac{|\GL_n(q)|}{2|\GL_{\frac{n}{2}}(q^2)|}=\frac{q^{\frac{n}{2}}}{2}\prod_{i=1}^{\frac{n}{2}}(q^{2i-1}-1)
>\frac{q^{\frac{n}{2}}}{2}\prod_{i=1}^{\frac{n}{2}}q^{2i-2}=\frac{q^{\frac{n^2}{4}}}{2}.
\]
In both cases we have $|G{:}M|>q^{\frac{n^2}{4}}/2$, so by Lemma~\ref{lem6} there is an absolute constant $c$ such that
\[
\frac{i_2(M)}{i_2(G)}<c|G{:}M|^{-\frac{2}{5}}<2^{\frac{2}{5}}cq^{-\frac{n^2}{10}}.
\]

\underline{Row~2.} $G=\PSU_n(q)$ with $n\geqslant3$ odd and $(n,q)\neq(3,2)$. Let $N=\C_{(q^{2n}-1)/(q^2-1)}\rtimes\C_n$ be the normalizer of the Singer cycle of $\PGL_n(q^2)$ containing $x$. Then from the argument for row~1 we see that $n$ is even, a contradiction.

\underline{Row~3.} $G=\PSU_n(q)$ with $n\geqslant4$ even and $(n,q)\neq(4,2)$. Let $S=\C_{q^{2(n-1)}-1}$ be a Singer cycle of $\GL_{n-1}(q^2)$ and $N=\C_{q^{2(n-1)}-1}\rtimes\C_{n-1}$ be the normalizer of $S$ in $\GL_{n-1}(q^2)$. For any $H\leqslant\GL_{n-1}(q^2)$ denote the projection of
\[
\left\{
\begin{pmatrix}
A&0\\
0&1
\end{pmatrix}
\middle|A\in H\right\}\cap\GU_n(q)
\]
into $\PGU_n(q)$ by $\overline{H}$. We have $\overline{S}=\C_{q^{n-1}+1}$ and $\overline{N}=\overline{S}.\C_{n-1}$. Since $\overline{S}$ contains a Sylow $r$-subgroup of $\PGU_n(q)$, we may assume without loss of generality that $x\in\overline{S}$. It follows that $\overline{N}\leqslant\Nor_{\PGU_n(q)}(\langle x\rangle)$, and by Lemma~\ref{lem7},
\[
|\Nor_G(\langle x\rangle)||\PGU_n(q)/G|=\frac{(n-1)(q^{n-1}+1)|\PGU_n(q)/G|}{\gcd(n,q+1)}=(n-1)(q^{n-1}+1)=|\overline{N}|.
\]
This leads to
\[
|\overline{N}|\leqslant|\Nor_{\PGU_n(q)}(\langle x\rangle)|\leqslant|\Nor_G(\langle x\rangle)||\PGU_n(q)/G|=|\overline{N}|
\]
and thus $\Nor_{\PGU_n(q)}(\langle x\rangle)=\overline{N}$. Consequently, $\alpha\in\overline{N}$ as $\alpha\in\Nor_{\PGU_n(q)}(\langle x\rangle)$. Hence $\alpha$ is an involution in $\overline{N}=\overline{S}.\C_{n-1}$ such that $x^\alpha=x^{-1}$, which implies that $n-1$ is even, a contradiction.

\underline{Row~4.} $G=\PSp_n(q)$ with $n\geqslant4$ even and $(n,q)\neq(4,2)$ or $(6,2)$. Let $N=\C_{(q^n-1)/(q-1)}\rtimes\C_n$ be the normalizer of the Singer cycle of $\PGL_n(q)$ containing $x$. In the argument for row~1 we have obtained that $\alpha$ is induced by some $\beta\in\GL_n(q)$ such that $\beta$ has no eigenspace of dimension larger than $n/2$ over $\bbF_q$. Then similarly as for row~1, we see from the proof of~\cite[Proposition~4.1]{LS1996} that $\Cen_G(\alpha)$ is contained in a maximal subgroup $M$ of $G$ such that $M$ lies in Table~\ref{tab7}. By virtue of~\cite[Propositions~4.1.19,~4.2.5,~4.2.10,~4.3.7~and~4.3.10]{KL1990} we can easily calculate $|G{:}M|$, which is listed in Table~\ref{tab7}. It follows that $|G{:}M|>q^{\frac{n^2}{8}+\frac{n}{4}}$, and so by Lemma~\ref{lem6} there is an absolute constant $c$ such that
\[
\frac{i_2(M)}{i_2(G)}<c|G{:}M|^{-\frac{2}{5}}<cq^{-\frac{n^2}{20}-\frac{n}{10}}.
\]

\begin{table}[htbp]
\caption{The pair $(M,|G{:}M|)$ for $G=\PSp_n(q)$ in the proof of Lemma~\ref{lem3}}\label{tab7}
\centering
\begin{tabular}{|l|l|l|l|}
\hline
Class & Type of $M$ & Conditions & $|G{:}M|$\\
\hline
$\calC_1$ & $\Pa_\frac{n}{2}$ & $q$ even & $\prod\limits_{i=1}^\frac{n}{2}(q^i+1)$\\
$\calC_2$ & $\GL_\frac{n}{2}(q).\C_2$ & $q$ odd & $\frac{q^{\frac{n^2}{8}+\frac{n}{4}}}{2}\prod\limits_{i=1}^\frac{n}{2}(q^i+1)$\\
$\calC_2$ & $\Sp_\frac{n}{2}(q)\wr\Sy_2$ & $n/2$ even, $q$ odd & $\frac{q^\frac{n^2}{8}}{2}\prod\limits_{i=1}^\frac{n}{4}\frac{q^{\frac{n}{2}+2i}-1}{q^{2i}-1}$\\
$\calC_3$ & $\GU_\frac{n}{2}(q).\C_2$ & $q$ odd & $\frac{q^{\frac{n^2}{8}+\frac{n}{4}}}{2}\prod\limits_{i=1}^\frac{n}{2}(q^i+(-1)^i)$\\
$\calC_3$ & $\Sp_\frac{n}{2}(q^2).\C_2$ & $n/2$ even, $q$ odd & $\frac{q^\frac{n^2}{8}}{2}\prod\limits_{i=1}^\frac{n}{4}(q^{4i-2}-1)$\\
\hline
\end{tabular}
\end{table}

\underline{Row~5.} $G=\POm_n(q)$ with $n\geqslant7$ odd and $q$ odd. Let $S=\C_{q^{n-1}-1}$ be a Singer cycle of $\GL_{n-1}(q)$ and $N=\C_{q^{n-1}-1}\rtimes\C_{n-1}$ be the normalizer of $S$ in $\GL_{n-1}(q)$. For any $H\leqslant\GL_{n-1}(q)$ denote the projection of
\[
\left\{
\begin{pmatrix}
A&0\\
0&1
\end{pmatrix}
\middle|A\in H\right\}\cap\GO_n(q)
\]
into $\PGO_n(q)$ by $\overline{H}$. We have $\overline{S}=\C_{q^{(n-1)/2}+1}$ and $\overline{N}=\overline{S}.\C_{n-1}$. Since $\overline{S}$ contains a Sylow $r$-subgroup of $\PGO_n(q)$, we may assume without loss of generality that $x\in\overline{S}$. It follows that $\overline{N}\leqslant\Nor_{\PGO_n(q)}(\langle x\rangle)$, and by Lemma~\ref{lem7},
\[
|\Nor_G(\langle x\rangle)||\PGO_n(q)/G|=\frac{(n-1)(q^\frac{n-1}{2}+1)|\PGO_n(q)/G|}{2}=(n-1)(q^\frac{n-1}{2}+1)=|\overline{N}|.
\]
This leads to
\[
|\overline{N}|\leqslant|\Nor_{\PGO_n(q)}(\langle x\rangle)|\leqslant|\Nor_G(\langle x\rangle)||\PGO_n(q)/G|=|\overline{N}|
\]
and thus $\Nor_{\PGO_n(q)}(\langle x\rangle)=\overline{N}$. Consequently, $\alpha\in\overline{N}$ as $\alpha\in\Nor_{\PGO_n(q)}(\langle x\rangle)$. Now $\alpha$ is an involution in $\overline{N}=\overline{S}.\C_{n-1}$ such that $x^\alpha=x^{-1}$. Similarly as for row~1 we obtain that $\alpha$ is induced by some $\beta\in\GL_{n-1}(q)$ such that $\beta$ has no eigenspace of dimension larger than $(n-1)/2$ over $\bbF_q$, which implies that $\alpha$ has no eigenspace of dimension larger than $(n+1)/2$ over $\bbF_q$. Then we see from the proof of~\cite[Proposition~4.1]{LS1996} that $\Cen_G(\alpha)$ is contained in the stabilizer $M$ in $G$ of a nonsingular subspace of dimension $(n-1)/2$ over $\bbF_q$. Hence
\[
|G{:}M|=\frac{|\GO_n(q)|}{|\GO_\frac{n-1}{2}^\varepsilon(q)||\GO_\frac{n+1}{2}(q)|}
=\frac{q^\frac{n^2-1}{8}(q^\frac{n-1}{2}+1)(q^\frac{n-1}{4}+\varepsilon1)}{2}\prod_{i=1}^\frac{n-5}{4}\frac{q^{\frac{n-1}{2}+2i}-1}{q^{2i}-1}
\]
if $n\equiv1\pmod{4}$, and
\[
|G{:}M|=\frac{|\GO_n(q)|}{|\GO_\frac{n+1}{2}^\varepsilon(q)||\GO_\frac{n-1}{2}(q)|}
=\frac{q^\frac{n^2-1}{8}(q^\frac{n+1}{4}+\varepsilon1)}{2}\prod_{i=1}^\frac{n-3}{4}\frac{q^{\frac{n+1}{2}+2i}-1}{q^{2i}-1}
\]
if $n\equiv3\pmod{4}$, where $\varepsilon=\pm$. In both cases we have $|G{:}M|>q^{\frac{n^2-5}{4}}/2$, so by Lemma~\ref{lem6} there is an absolute constant $c$ such that
\[
\frac{i_2(M)}{i_2(G)}<c|G{:}M|^{-\frac{2}{5}}<2^{\frac{2}{5}}cq^{-\frac{n^2}{10}+\frac{1}{2}}.
\]

\underline{Row~6.} $G=\POm_n^+(q)$ with $n\geqslant8$ even. Let $S=\C_{q^{n-2}-1}$ be a Singer cycle of $\GL_{n-2}(q)$ and $N=(\C_{q^{n-2}-1}\rtimes\C_{n-2})\times\mathrm{GO}_2^-(q)$ be the normalizer of $S$ in $\GL_{n-2}(q)\times\mathrm{GO}_2^-(q)$. Note that $N\cap\mathrm{GO}_n^+(q)=(\C_{q^{(n-2)/2}+1}\rtimes\C_{n-2})\times\mathrm{GO}_2^-(q)$. For any $H\leqslant\GL_n(q)$ denote the projection of $H\cap\mathrm{GO}_n^+(q)$ into $\PGO_n^+(q)$ by $\overline{H}$. We then have $\overline{S}=\C_{q^{(n-2)/2}+1}$ and
\[
|\overline{N}|=\frac{|N\cap\mathrm{GO}_n^+(q)|}{q-1}=2(n-2)(q^{\frac{n}{2}-1}+1)(q+1)
\]
Since $\overline{S}$ contains a Sylow $r$-subgroup of $\PGO_n^+(q)$, we may assume without loss of generality that $x\in\overline{S}$. It follows that $\overline{N}\leqslant\Nor_{\PGO_n^+(q)}(\langle x\rangle)$, and by Lemma~\ref{lem7},
\begin{align*}
|\Nor_G(\langle x\rangle)||\PGO_n^+(q)/G|&=\frac{2(n-2)(q^{\frac{n}{2}-1}+1)(q+1)|\PGO_n^+(q)/G|}{{\gcd(2,q-1)^2|\PSO_n^+(q)/G|}}\\
&=2(n-2)(q^{\frac{n}{2}-1}+1)(q+1)=|\overline{N}|.
\end{align*}
This leads to
\[
|\overline{N}|\leqslant|\Nor_{\PGO_n^+(q)}(\langle x\rangle)|\leqslant|\Nor_G(\langle x\rangle)||\PGO_n^+(q)/G|=|\overline{N}|
\]
and thus $\Nor_{\PGO_n^+(q)}(\langle x\rangle)=\overline{N}$. Consequently, $\alpha\in\overline{N}$ as $\alpha\in\Nor_{\PGO_n^+(q)}(\langle x\rangle)$. Now $\alpha$ is an involution in $\overline{N}$ such that $x^\alpha=x^{-1}$. Similarly as for row~1 we obtain that $\alpha$ is induced by some $(\beta_1,\beta_2)\in\GL_{n-2}(q)\times\GL_2(q)$ such that $\beta_1$ has no eigenspace of dimension larger than $(n-2)/2$ over $\bbF_q$, which implies that $\alpha$ has no eigenspace of dimension larger than $(n+2)/2$ over $\bbF_q$. Then similarly as for row~1, we see from the proof of~\cite[Proposition~4.1]{LS1996} that $\Cen_G(\alpha)$ is contained in a maximal subgroup $M$ of $G$ such that $M$ lies in Table~\ref{tab10}, where $\N_i^+$, $\N_i^-$ and $\N_i$ are as defined in~\cite[Page~25]{LPS1990}. Appealing to~\cite[Propositions~4.1.6,~4.1.20,~4.2.7,~4.2.11,~4.2.14,~4.2.16,~4.3.14,~4.3.18 and~4.3.20]{KL1990} we can easily calculate $|G{:}M|$, which is listed in Table~\ref{tab10}. It follows that $|G{:}M|>q^{\frac{n^2}{8}-\frac{n}{4}}$, and so by Lemma~\ref{lem6} there is an absolute constant $c$ such that
\[
\frac{i_2(M)}{i_2(G)}<c|G{:}M|^{-\frac{2}{5}}<cq^{-\frac{n^2}{20}+\frac{n}{10}}.
\]

\begin{table}[htbp]
\caption{The pair $(M,|G{:}M|)$ for $G=\POm_n^+(q)$ in the proof of Lemma~\ref{lem3}}\label{tab10}
\centering
\begin{tabular}{|l|l|l|l|}
\hline
Class & Type of $M$ & Conditions & $|G{:}M|$\\
\hline
$\calC_1$ & $\N_{\frac{n}{2}-1}^\varepsilon$ & $n/2$ odd, $q$ odd, $\varepsilon=\pm$ & $\frac{q^{\frac{n^2}{8}-\frac{1}{2}}(q^\frac{n}{2}-1)(q^\frac{n+2}{4}+\varepsilon1)}{2(q^\frac{n-2}{4}-\varepsilon1)}
\prod\limits_{i=1}^\frac{n-6}{4}\frac{q^{\frac{n}{2}+1+2i}-1}{q^{2i}-1}$\\
$\calC_1$ & $\N_{\frac{n}{2}-1}$ & $n/2$ even, $q$ odd & $\frac{q^{\frac{n^2}{8}-1}(q^\frac{n}{2}-1)}{2}\prod\limits_{i=1}^{\frac{n}{4}-1}\frac{q^{\frac{n}{2}+2i}-1}{q^{2i}-1}$\\
$\calC_1$ & $\Pa_{\frac{n}{2}-1}$ & $q$ even & $\frac{q^\frac{n}{2}-1}{q-1}\prod\limits_{i=1}^{\frac{n}{2}-1}(q^i+1)$\\
$\calC_1$ & $\Pa_\frac{n}{2}$ & $q$ even & $\prod\limits_{i=1}^{\frac{n}{2}-1}(q^i+1)$\\
$\calC_2$ & $\GL_\frac{n}{2}(q).\C_2$ & $q$ odd & $\frac{q^{\frac{n^2}{8}-\frac{n}{4}}}{\gcd(2,n/2)}\prod\limits_{i=1}^{\frac{n}{2}-1}(q^i+1)$\\
$\calC_2$ & $\GO_\frac{n}{2}^\varepsilon(q)\wr\Sy_2$ & $n/2$ even, $q$ odd, $\varepsilon=\pm$ & $\frac{q^\frac{n^2}{8}(q^\frac{n}{4}+\varepsilon1)^2}{4}\prod\limits_{i=1}^{\frac{n}{4}-1}\frac{q^{\frac{n}{2}+2i}-1}{q^{2i}-1}$\\
$\calC_2$ & $\GO_\frac{n}{2}(q)\wr\Sy_2$ & $n/2$ odd, $q\equiv1\pmod{4}$ & $\frac{q^{\frac{n^2}{8}-\frac{1}{2}}(q^\frac{n}{2}-1)}{4}\prod\limits_{i=1}^{\frac{n-2}{4}}\frac{q^{\frac{n}{2}-1+2i}-1}{q^{2i}-1}$\\
$\calC_2$ & $\GO_\frac{n}{2}(q)^2$ & $n/2$ odd, $q\equiv3\pmod{4}$ & $\frac{q^{\frac{n^2}{8}-\frac{1}{2}}(q^\frac{n}{2}-1)}{2}\prod\limits_{i=1}^{\frac{n-2}{4}}\frac{q^{\frac{n}{2}-1+2i}-1}{q^{2i}-1}$\\
$\calC_3$ & $\GO_\frac{n}{2}^+(q^2).\C_2$ & $n/2$ even, $q$ odd & $\frac{q^\frac{n^2}{8}}{4}\prod\limits_{i=1}^\frac{n}{4}(q^{4i-2}-1)$\\
$\calC_3$ & $\GU_\frac{n}{2}(q).\C_2$ & $n/2$ even, $q$ odd & $\frac{q^{\frac{n^2}{8}-\frac{n}{4}}}{2}\prod\limits_{i=1}^{\frac{n}{2}-1}(q^i+(-1)^i)$\\
$\calC_3$ & $\GO_\frac{n}{2}(q^2).\C_2$ & $n/2$ odd, $q$ odd & $\frac{q^{\frac{n^2}{8}-\frac{1}{2}}(q^\frac{n}{2}-1)}{\gcd(4,q-1)}\prod\limits_{i=1}^\frac{n-2}{4}(q^{4i-2}-1)$\\
\hline
\end{tabular}
\end{table}

\underline{Row~7.} $G=\POm_n^-(q)$ with $n\geqslant8$ even. Let $N=\C_{(q^n-1)/(q-1)}\rtimes\C_n$ be the normalizer of the Singer cycle of $\PGL_n(q)$ containing $x$. In the argument for row~1 we have obtained that $\alpha$ is induced by some $\beta\in\GL_n(q)$ such that $\beta$ has no eigenspace of dimension larger than $n/2$ over $\bbF_q$. Then similarly as for row~1, we see from the proof of~\cite[Proposition~4.1]{LS1996} that $\Cen_G(\alpha)$ is contained in a maximal subgroup $M$ of $G$ such that $M$ lies in Table~\ref{tab9}. Thus appealing to~\cite[Propositions~4.1.20,~4.2.14,~4.2.16,~4.3.16,~4.3.18 and~4.3.20]{KL1990} we can easily calculate $|G{:}M|$, which is listed in Table~\ref{tab9}. It follows that $|G{:}M|>q^{\frac{n^2}{8}-\frac{n}{4}}$, and so by Lemma~\ref{lem6} there is an absolute constant $c$ such that
\[
\frac{i_2(M)}{i_2(G)}<c|G{:}M|^{-\frac{2}{5}}<cq^{-\frac{n^2}{20}+\frac{n}{10}}.
\]
This completes the proof.

\begin{table}[htbp]
\caption{The pair $(M,|G{:}M|)$ for $G=\POm_n^-(q)$ in the proof of Lemma~\ref{lem3}}\label{tab9}
\centering
\begin{tabular}{|l|l|l|l|}
\hline
Class & Type of $M$ & Conditions & $|G{:}M|$\\
\hline
$\calC_1$ & $\Pa_\frac{n}{2}$ & $q$ even & $\frac{q^\frac{n}{2}+1}{q^\frac{n}{2}-1}\prod\limits_{i=1}^{\frac{n}{2}-1}(q^i+1)$\\
$\calC_2$ & $\GO_\frac{n}{2}(q)\wr\Sy_2$ & $n/2$ odd, $q\equiv3\pmod{4}$ & $\frac{q^{\frac{n^2}{8}-\frac{1}{2}}(q^\frac{n}{2}+1)}{4}\prod\limits_{i=1}^{\frac{n-2}{4}}\frac{q^{\frac{n}{2}-1+2i}-1}{q^{2i}-1}$\\
$\calC_2$ & $\GO_\frac{n}{2}(q)^2$ & $n/2$ odd, $q\equiv1\pmod{4}$ & $\frac{q^{\frac{n^2}{8}-\frac{1}{2}}(q^\frac{n}{2}+1)}{2}\prod\limits_{i=1}^{\frac{n-2}{4}}\frac{q^{\frac{n}{2}-1+2i}-1}{q^{2i}-1}$\\
$\calC_3$ & $\GO_\frac{n}{2}^-(q^2).\C_2$ & $n/2$ even, $q$ odd & $\frac{q^\frac{n^2}{8}}{2}\prod\limits_{i=1}^\frac{n}{4}(q^{4i-2}-1)$\\
$\calC_3$ & $\GU_\frac{n}{2}(q).\C_2$ & $n/2$ odd, $q$ odd & $q^{\frac{n^2}{8}-\frac{n}{4}}\prod\limits_{i=1}^{\frac{n}{2}-1}(q^i+(-1)^i)$\\
$\calC_3$ & $\GO_\frac{n}{2}(q^2).\C_2$ & $n/2$ odd, $q$ odd & $\frac{q^{\frac{n^2}{8}-\frac{1}{2}}(q^\frac{n}{2}+1)}{\gcd(4,q+1)}\prod\limits_{i=1}^\frac{n-2}{4}(q^{4i-2}-1)$\\
\hline
\end{tabular}
\end{table}
\end{proof}

For a group $G$ and a subset $S$ of $G$, denote by $\Aut(G,S)$ the set of automorphisms of $G$ that fix $S$ setwise. Now we are ready to prove Theorem~\ref{thm1}.

\vskip0.1in
\noindent\textit{Proof of Theorem~$\ref{thm1}$:} Let $G$ be a finite simple classical group with natural module $V$ and $x$ be an element of order $r\in\ppd(p,ef)$ in $G$, where $G$ and $e$ are given in Table~$\ref{tab4}$ with $q=p^f$ and $p$ prime. Denote the sets of involutions of $G$ and $\Aut(G)$ by $I$ and $J$, respectively, and write
\[
K=\{y\in I\mid G=\langle x,y\rangle\}\quad\text{and}\quad L=\{y\in I\mid\Aut(G,\{x,x^{-1},y\})=1\}.
\]
Without loss of generality we assume that $q^n$ is large enough such that by~\cite[Theorem~5.2.2]{KL1990}, $G$ has no proper subgroup of index at most $47$. Then as a consequence of~\cite[Theorem~3.3]{Godsil1983}, for every $y\in I$, $\Cay(G,\{x,x^{-1},y\})$ is a GRR of $G$ if and only if $y\in K\cap L$.

Consider an arbitrary element $y$ of $K\setminus L$. There exists a nontrivial $\alpha\in\Aut(G)$ such that $\alpha$ fixes $\{x,x^{-1},y\}$ setwise. Since $x$ and $x^{-1}$ both have order $r>e>2$ while $y$ has order $2$, it follows that $x^\alpha=x^{-1}$ and $y^\alpha=y$. In particular, $\alpha^2$ fixes both $x$ and $y$. Then as $G=\langle x,y\rangle$, we conclude that $\alpha^2=1$ and so $\alpha\in J$. This shows that
\[
K\setminus L\subseteq\bigcup_{\substack{\alpha\in J\\ x^\alpha=x^{-1}}}(I\cap\Cen_G(\alpha)).
\]
Using Lemmas~\ref{lem4}, \ref{lem5}~and~\ref{lem3} one can verify that, for each $\alpha\in J$ with $x^\alpha=x^{-1}$, we have $i_2(\Cen_G(\alpha))<u(G)i_2(G)$ for $u(G)$ in Table~\ref{tab15}, where $C$ is an absolute constant. Accordingly,
\[
|K\setminus L|\leqslant\sum_{\substack{\alpha\in J\\ x^\alpha=x^{-1}}}i_2(\Cen_G(\alpha))
<\sum_{\substack{\alpha\in J\\ x^\alpha=x^{-1}}}u(G)i_2(G)\leqslant u(G)i_2(G)|J\cap\Nor_{\Aut(G)}(\langle x\rangle)|.
\]
Since every involution in $\Nor_{\Aut(G)}(\langle x\rangle)$ projects to an involution or the identity in $\Aut(G)/(\Aut(G)\cap\PGL(V))$ and $i_2(\Aut(G)/(\Aut(G)\cap\PGL(V)))+1\leqslant4$ (see for example~\cite[\S1.7]{BHR2013}), we deduce that
\[
|J\cap\Nor_{\Aut(G)}(\langle x\rangle)|\leqslant4|\Nor_{\Aut(G)\cap\PGL(V)}(\langle x\rangle)|
\leqslant\frac{4|\Nor_G(\langle x\rangle)||\Aut(G)\cap\PGL(V)|}{|G|}.
\]
This implies by Lemma~\ref{lem7} that $|J\cap\Nor_{\Aut(G)}(\langle x\rangle)|<v(G)$ for $v(G)$ in Table~\ref{tab15}. It follows that
\[
|K\setminus L|<u(G)i_2(G)|J\cap\Nor_{\Aut(G)}(\langle x\rangle)|<u(G)i_2(G)v(G)
\]
and so
\[
\frac{|K\cap L|}{|I|}=\frac{|K|}{|I|}-\frac{|K\setminus L|}{|I|}>\frac{|K|}{|I|}-u(G)v(G).
\]
By Proposition~\ref{prop1}, $|K|/|I|$ tends to $1$ as $q^n$ tends to infinity. Moreover, noting that $n\leqslant\log_2(q^n)$, it is clear from Table~\ref{tab15} that $u(G)v(G)$ tends to $0$ as $q^n$ tends to infinity. Thus we conclude that $|K\cap L|/|I|$ tends to $1$ as $q^n$ tends to infinity, which means that the probability of a random $y\in I$ making $\Cay(G,\{x,x^{-1},y\})$ a GRR of $G$ tends to $1$ as $q^n$ tends to infinity.

\begin{table}[htbp]
\caption{The tuple $(G,e,u(G),v(G))$ in the proof of Theorem~\ref{thm1}}\label{tab15}
\centering
\begin{tabular}{|l|l|l|l|l|}
\hline
$G$ & Conditions & $e$ & $u(G)$ & $v(G)$\\
\hline
$\PSL_n(q)$ & $n\geqslant9$ & $n$ & $Cq^{-\frac{91n}{90}+1}$ & $8nq^{n-1}$\\
$\PSU_n(q)$ & $n\geqslant5$ odd & $2n$ & $Cq^{-\frac{11n}{10}+1}$ & $4nq^{n-1}$\\
$\PSU_n(q)$ & $n\geqslant6$ even & $2(n-1)$ & $Cq^{-\frac{4n}{3}+1}$ & $4nq^{n-1}$\\
$\PSp_n(q)$ & $n\geqslant10$ even & $n$ & $Cq^{-\frac{3n}{5}}$ & $6nq^\frac{n}{2}$\\
$\POm_n(q)$ & $n\geqslant9$ odd, $q$ odd & $n-1$ & $Cq^{-\frac{9n}{10}+\frac{1}{2}}$ & $6nq^{\frac{n}{2}-\frac{1}{2}}$\\
$\POm_n^+(q)$ & $n\geqslant14$ even & $n-2$ & $Cq^{-\frac{3n}{5}}$ & $16nq^\frac{n}{2}$\\
$\POm_n^-(q)$ & $n\geqslant14$ even & $n$ & $Cq^{-\frac{3n}{5}}$ & $6nq^\frac{n}{2}$\\
\hline
\end{tabular}
\end{table}
\qed

\vskip0.1in
\noindent\textsc{Acknowledgement.} The work on this paper was done when the author was a research associate at the University of Western Australia supported by the Australian Research Council Discovery Project DP150101066. The author would like to thank Prof.~Michael Giudici for his help and the anonymous referee for careful reading and valuable suggestions.

\end{document}